\documentclass[twoside, 11pt]{article}
\usepackage[top=.8
in, bottom=1in, left=1in, right=1in]{geometry}

\usepackage{enumerate,float,mathtools}

\usepackage{color, section, amsthm, textcase, setspace, amssymb, lineno, 
amsmath, amssymb, amsfonts, latexsym, fancyhdr, longtable, ulem}

\usepackage{section, amsthm, textcase, setspace, amssymb, lineno, amsmath, amssymb, amsfonts, latexsym, fancyhdr, longtable, ulem}

\usepackage{float}
\usepackage{mathrsfs}

\usepackage[most]{tcolorbox}

\usepackage{tikz,tikz-3dplot,tkz-graph}
\usetikzlibrary{decorations.markings}
\usetikzlibrary{arrows.meta}
\tikzstyle{vertex}=[circle, draw, inner sep=0pt, minimum size=13pt]
\newcommand{\vertex}{\node[vertex]}

\newtheorem{theorem}{Theorem}[section]

\theoremstyle{definition}
\newtheorem{definition}[theorem]{Definition}
\newtheorem{example}[theorem]{Example}
\newtheorem{remark}[theorem]{Remark}

\newcommand{\mf}{\mathfrak}
\newcommand{\wh}{\widehat}
\newcommand{\ul}{\underline}

\newcommand{\dd}{\hspace{.1cm}|\hspace{.1cm}}
\newcommand{\ind}{{\rm ind \hspace{.1cm}}}

\newcommand{\g}{{\mathfrak{g}}}
\newcommand{\A}{{\rm A}}
\newcommand{\B}{{\rm B}}
\newcommand{\C}{{\rm C}}
\newcommand{\D}{{\rm D}}

\newcommand{\G}{{\rm G}}
\newcommand{\I}{{\rm I}}
\newcommand{\II}{{\rm II}}
\newcommand{\III}{{\rm III}}
\newcommand{\ad}{{\rm ad \hspace{.05cm}}}

\newcommand{\lf}{\left\lfloor}
\newcommand{\rf}{\right\rfloor}
\newcommand{\CC}{\mathbb C}
\newcommand{\Z}{\mathbb Z}
\newcommand{\R}{\mathbb R}

\begin{document}

\title{\bf Seaweed algebras}
\author{Alex Cameron$^*$, Vincent E. Coll, Jr.$^{**}$, Nicholas Mayers$^{\dagger}$, and Nicholas Russoniello$^{\dagger\dagger}$}
\maketitle
\noindent
\textit{$^*$Department of Mathematics, Muhlenberg College, Allentown, PA: alexcameron@muhlenberg.edu }\\
\textit{$^{**}$Department of Mathematics, Lehigh University, Bethlehem, PA:  vec208@lehigh.edu}\\
\textit{$^{\dagger}$Department of Mathematics, Lehigh University, Bethlehem, PA: nwm5095@lehigh.edu}\\
\textit{$^{\dagger\dagger}$Department of Mathematics, Lehigh University, Bethlehem, PA: nvr217@lehigh.edu}\\
\begin{center}
Vincent E. Coll, Jr. is the corresponding author\\
vec208@lehigh.edu
\end{center}


\begin{abstract}\noindent
The index of a Lie algebra is an important algebraic invariant, but it is notoriously difficult to compute. However, for the suggestively-named seaweed algebras, the computation of the index can be reduced to a combinatorial formula based on the connected components of a ``meander":  a planar graph associated with the algebra. Our index analysis on seaweed algebras requires only basic linear and abstract algebra. Indeed, the main goal of this survey-type article is to introduce a broader audience to seaweed algebras with minimal appeal to specialized language and notation from Lie theory. This said, we present several results that do not appear elsewhere and do appeal to more advanced language in the Introduction to provide added context.
\end{abstract}






\section{Introduction}




Lie theory has a rich and pedigreed past and was initiated in the late nineteenth century by the Norwegian mathematician Sophus Lie (1842-1899). Lie's initial focus was to extend, by analogy, the Galois theory of roots of polynomials -- where finite permutation groups provided insight into the computability of roots by radicals -- to develop a new theory using ``continuous" groups to study the solutions of differential equations \textbf{\cite{Lie1, Lie2}}. 

In modern terminology, Lie's groups are called \textit{Lie groups}, and they are smooth manifolds with a compatible group structure.  That is, both the multiplication and inverse maps of the group are smooth operations on the manifold.
 Lie groups are of central importance in many fields of mathematics, but they play a fundamental role in physics -- appearing as symmetry groups of physical systems.\footnote{To illustrate the notion of a symmetry, consider a mechanical system governed by a differential equation of one variable, say $x$.  If the solutions to the differential equation are invariant when $x$ is replaced by $f(x)$, then $f$ is called a \textit{symmetry} of the system. In \textbf{\cite{HWeyl}}, H. Weyl states it colloquially this way:  ``A symmetry exists if you do something and it does not make a difference".    It is a fundamental tenet of physics, established mathematically in 1915 by Emmy Noether, that every continuous or differentiable symmetry of the action of a physical system has a corresponding conservation law.}




The essential feature of Lie theory is that one may associate with any Lie group $G$ a Lie algebra $\mathfrak{g}$, where one thinks of the Lie algebra as the tangent space to the group's identity.  Thus, Lie algebras are, in the language of Lie, ``infinitesimal symmetries". 
The crucial, and rather surprising, fact is that $G$ is almost\footnote{The connected component of the identity is completely determined, but the existence of other connected components is not.  For example, $O(n)$ and $SO(n)$ have the same Lie algebra, but the Lie groups are not the same. } completely determined by its Lie algebra.  And so, for many purposes, questions concerning the complicated nonlinear structure of $G$ can be
translated into questions about $\mathfrak{g}$ -- where otherwise intractable computations may become problems in linear algebra.

More formally, a Lie algebra $\mf{g}$ is a vector space together with a bilinear, skew-symmetric product, $[-,-]: \mf{g}\rightarrow \mf{g}$,  called the \textit{bracket}, which satisfies the 
 the Jacobi identity:  $[x,[y,z]] + [z,[x,y]] + [y,[z,x]] =0$.  A base example is the matrix Lie algebra $\mathfrak{gl}(n)$, which is the vector space of $n \times n$ matrices over the field of complex numbers where the Lie structure is given by the \textit{commutator bracket}:  $[A,B] = AB-BA$.  
  
So, it is that we come to the study of Lie algebras proper and are left with the question of what computable invariants might allow us to distinguish between two Lie algebras.  
Here, we focus on the ``index" of a Lie algebra $(\mathfrak{g},[-,-])$.  This well-studied algebraic invariant was introduced by Dixmier (\textbf{\cite{Dix}}, 1974) and is defined by 
\begin{eqnarray}\label{Frob}
\ind \g=\min_{f\in \mathfrak{g}^*} \dim  (\ker (B_f)),
\end{eqnarray}
where $f$ is an element of the linear dual $\mathfrak{g}^*$, and $B_f$ is the associated skew-symmetric \textit{Kirillov form} defined by 
\[
B_f(x,y)=f([x,y]), \textit{ for all }~ x,y\in\g. \]


The index may be regarded as a generalization of the rank of Lie algebra, for if $\mathfrak{q}$ is reductive, then $\ind \mathfrak{q} = \text{rk} ~\mathfrak{q}$; and while the definition of index given above is nicely linear algebraic, we are left with the question of how it might be computed.  This overly broad question can be significantly mollified by focusing on a certain class of matrix Lie algebras called \textit{seaweed algebras} (``seaweeds'', going forward).  Associated with each seaweed is a planar graph called a \textit{meander}.   Remarkably, the index of a seaweed can be computed by an elementary combinatorial formula in terms of the number of paths and cycles in its associated meander.  These combinatorial formulas, and the ``linear greatest common divisor index formulas" that follow from them, form the coda of our survey article here. 

Seaweed algebras, along with their suggestive name, were first introduced by Dergachev and A. Kirillov (\textbf{\cite{DK}}, 2000), where they defined seaweeds as subalgebras of $\mathfrak{gl}(n)$ preserving a pair of flags of subspaces of $\CC^n$.  Shortly thereafter, Panyushev (\textbf{\cite{Panyushev1}}, 2001) provided a more general definition suited for arbitrary reductive Lie algebras. To ease exposition, we will restrict our attention to seaweed subalgebras of the classical simple Lie algebras.

  The simple Lie algebras have an old, well-known, and complete classification -- owing primarily to the works of Killing, E. Cartan, and Dynkin.\footnote{The journey from Lie's initial observations to a complete classification of complex semisimple Lie algebras  spans nineteenth and twentieth century algebraic theory.  The classification has its roots in the work of W. Killing (1847 - 1923) -- who discovered Lie algebras independently from Lie -- but  was not fully realized until the remarkable 1894 thesis
\textbf{\cite{Cartan}} of Lie's student, E. Cartan (1869 - 1951).
In 1947, the then twenty-two-year old E. Dynkin (1924 - 2014) began publishing a series of papers \textbf{\cite{Dynkin0,Dynkin1, Dynkin2}}.   In the first of these, he introduced the so-called ``Dynkin diagrams'', greatly simplifying the presentation of Cartan's classification, thereby recasting the theory in its near current form. The early history (1869-1926) of the classification story has been catalogued in Hawkins's lengthy treatise \textbf{\cite{Hawkins}.}  See also  \textbf{\cite{Rowe}} for a summary narrative.}  To summarize this work, we quote only part of the final result. 


\begin{theorem}\label{classification}
With five exceptions, every simple, complex Lie algebra is isomorphic to one in the following four classical families -- of types A, B, C, and D -- each of which can be represented as a matrix subalgebra of $\mathfrak{gl}(n)$ as follows. \textup(The subscripts indicate the rank of the algebra.\textup)
\begin{itemize}
\item $A_n=\mathfrak{sl}(n+1)$ -- the special linear Lie algebra

$$\{H\in \mathfrak{gl}(n) : \text{tr}(H)=0 \},$$


\item $B_n=\mathfrak{so}(2n+1)$ -- the odd-dimensional special orthogonal Lie algebra
$$\left\{\begin{bmatrix} E & F \\G & -\wh{E} \end{bmatrix}: F=-\wh{F}, G=-\wh{G} \right\},$$
\item $C_n=\mathfrak{sp}(2n)$ -- the even-dimensional symplectic Lie algebra
$$\left\{\begin{bmatrix} H & I \\J & -\wh{H} \end{bmatrix}: I=\wh{I}, J=\wh{J} \right\},$$
\item $D_n=\mathfrak{so}(2n)$ -- the even-dimensional special orthogonal Lie algebra, $(n>1)$
$$\left\{\begin{bmatrix} H & I \\ J & -\wh{H}\end{bmatrix}: I=-\wh{I}, J=-\wh{J} \right\},$$
\end{itemize}
where $G, H, I,$ and $J$ are $n\times n$ matrices, $E$ is an $(n+1) \times n$ matrix, $F$ is an $(n+1) \times (n+1)$ matrix, and $\wh{E}$ is the transpose of 
$E$ with respect to the antidiagonal.
\end{theorem}

\noindent
We will see examples of members of each of these families as we examine each of the classical types in course.

A significant by-product of our investigation is the ability to construct Lie algebras of index zero.  These algebras are called \textit{Frobenius} and were introduced into the literature by Ooms (\textbf{\cite{Ooms}}, 1980) to answer a question of Jacobson about which ad-algebraic Lie algebras have a primitive universal enveloping algebra. Frobenius Lie algebras are also of interest to those working in deformation theory.  A Lie algebra $\g$ is Frobenius precisely when there exits $f\in \mathfrak{g}^*$ such that the Kirillov form is non-degenerate.  In this case, let $[f]$ be the matrix of $B_f(-,-)$ relative to some basis $\{x_1,\ldots,x_n\}$ of $\mathfrak{g}$.  Belavin and Drinfel'd  (\textbf{\cite{BelDrin}}, 1982) showed that 

$$
\sum_{i,j}[f]_{ij}^{-1}x_i\wedge x_j
$$

\noindent
is the infinitesimal of a \textit{Universal Deformation Formula} based on $\mathfrak{g}.$  A Universal Deformation Formula -- nowadays called a \textit{Drinfel'd twist} -- based on $\mathfrak{g}$ can be used to deform the universal enveloping algebra of $\mathfrak{g}$ and also any function space on any Lie group which contains $\mathfrak{g}$ in its Lie algebra of derivations. See \textbf{\cite{Coll1}}.


The organization of the paper is as follows.  In Section~\ref{two}, we introduce the notion of a seaweed subalgebra of $\mathfrak{gl}(n)$ -- along with its meander graph -- and find that such a  seaweed can never be Frobenius (see Theorem \ref{general linear}).  Therefore,  
to find Frobenius Lie algebras, we look to seaweeds contained in the Lie algebras of classical type -- the passage to which involves the imposition of certain algebraic conditions (see Theorem \ref{classification}). 
With an eye toward finding and classifying certain families of Frobenius seaweeds, we then conduct an index analysis in each of the classical types:  Section~\ref{type A} deals with type A, Section \ref{type C} with type C, Section \ref{type B} with type B (the results for type B following \textit{mutatis mutandis} from type C), and the more involved type-D results are treated in Section~\ref{type D} and are reliant on all the type A, B, and C results that  precede it.  The Euler totient function makes a critical appearance in Theorems \ref{HomotopyTypeH1} and  \ref{HomotopyTypeH11}, which are the -- chronological and conceptual -- culmination of this article.

\begin{remark}  In this article we will be content to state without proof many of the index results in types A, B, and C. However, we will go into some detail in type D where certain proofs and constructions appear for the first time.
See, for example, Theorem \ref{thm:mod}.
\end{remark}
\section{Seaweeds and meanders in $\mathfrak{gl}(n)$} \label{two}

Seaweeds in $\mathfrak{gl}(n)$ can be reckoned in ``standard form''\footnote{Following \textbf{\cite{DK}}, seaweeds are more usually defined as matrix algebras preserving a pair of flags of subspaces of $\CC^n$. However, since every seaweed is conjugate to one in standard form, we have incorporated the standard form directly into the definition.  More generally, a basis-free definition reckons a seaweed subalgebra of a simple Lie algebra $\mathfrak{g}$ as the intersection of two parabolic algebras whose sum is $\mathfrak{g}$ (see \textbf{\cite{Panyushev1}}).  For this reason, seaweeds have elsewhere been called \textit{biparabolic} (see \textbf{\cite{Ammari}}).  
We do not require the latter definition for our present discussion, although there is a short note regarding this point at the end of this article.  See the Epilogue. }
as matrix algebras whose underlying vector space is contained in $M_n(\CC)$ 
as follows.  To set the notation, let $\mathscr{L}$ and $\mathscr{U}$ denote, respectively, the lower and upper triangular matrices in $M_n(\CC)$.  Now, fix two compositions of $n$, $\ul{a}=(a_1,\dots,a_m)$ and $\ul{b}=(b_1,\dots,b_t)$. Next, let $X_{\ul{a}}$ be the set of block-diagonal matrices whose blocks have sizes 
\begin{eqnarray}\label{block1}
a_1\times a_1,\dots,a_m\times a_m
\end{eqnarray}
\noindent
and similarly for $X_{\ul{b}}$.  Finally, define the vector subspace of $M_n(\CC)$ to be the span over $\CC$
of 

$$(X_{\ul{a}} \cap \mathscr{L}) ~\cup~ (X_{\ul{b}} \cap \mathscr{U}).$$  This subspace is closed under matrix multiplication and may be given a Lie algebra structure by taking the commutator bracket.  The resulting Lie algebra is called a \textit{seaweed algebra} or simply a ``seaweed'' and is denoted by 
$\mathfrak{p}_n (\ul{a}, \ul{b})$; or, when we want to emphasize the \textit{parts} of the defining compositions, $\mathfrak{p}_n \frac{a_1|\cdots|a_m}{b_1|\cdots|b_t}$. 
Note that the block decomposition in (\ref{block1}) gives the seaweed its distinctive seaweed ``shape".  
See Example \ref{construction}.

\begin{example}\label{construction}
Consider the seaweed $\mf{p}_5\frac{4|1}{2|1|2}$. Figure \ref{lower} illustrates the matrix form of
$X_{(4,1)} \cap \mathscr{L}$,
and Figure \ref{upper} illustrates the matrix form of
$X_{(2,1,2)} \cap \mathscr{U}$.  
Figure \ref{seaweed} illustrates the span over $\CC$ of
$(X_{(4,1)} \cap \mathscr{L}) ~\cup~ (X_{(2,1,2)} \cap \mathscr{U})$.  
We call the locations of potentially nonzero complex entries in the seaweed $\textit{admissible locations},$ and these are indicated by $*$'s in the matrix representation. The vector space $(X_{(4,1)} \cap \mathscr{L}) ~\cup~ (X_{(2,1,2)} \cap \mathscr{U})$ is an associative algebra under standard matrix multiplication and becomes a Lie algebra using the \textit{commutator bracket}  $[x,y]=xy-yx$. 
It is worth noting that it is not \textit{a priori} clear that the product of two matrices of a certain seaweed shape is once again a seaweed algebra with the same seaweed shape  -- but it is --  thereby assuring that the commutator bracket is well-defined.

\begin{figure}[H]

\[\begin{tikzpicture}[scale=.53]
\draw (0,0) -- (0,5);
\draw (0,5) -- (5,5);
\draw (5,5) -- (5,0);
\draw (5,0) -- (0,0);

\draw [line width=3](0,5) -- (4,5);
\draw [line width=3](4,5) -- (4,1);
\draw [line width=3](5,1) -- (5,0);
\draw [line width=3](5,1) -- (4,1);

\draw [line width=3](0,5) -- (0,1);
\draw [line width=3](0,1) -- (4,1);
\draw [line width=3](4,1) -- (4,0);
\draw [line width=3](4,0) -- (5,0);

\draw [dotted] (0,5) -- (5,0);

\node at (0.5,4.4) {{\large *}};
\node at (1.5,4.4) {{\large *}};
\node at (2.5,4.4) {{\large *}};
\node at (3.5,4.4) {{\large *}};
\node at (0.5,3.4) {{\large *}};
\node at (1.5,3.4) {{\large *}};
\node at (2.5,3.4) {{\large *}};
\node at (3.5,3.4) {{\large *}};
\node at (0.5,2.4) {{\large *}};
\node at (1.5,2.4) {{\large *}};
\node at (2.5,2.4) {{\large *}};
\node at (3.5,2.4) {{\large *}};
\node at (0.5,1.4) {{\large *}};
\node at (1.5,1.4) {{\large *}};
\node at (2.5,1.4) {{\large *}};
\node at (3.5,1.4) {{\large *}};
\node at (4.5,0.4) {{\large *}};

\node at (-.35,3) {4};
\node at (3.7,0.5) {1};


\end{tikzpicture}
\hspace{1cm}
\begin{tikzpicture}[scale=.53]
\draw (0,0) -- (0,0);

\draw (.6,2.8) -- (1,2.5);
\draw (0,2.5) -- (1,2.5);
\draw (.6,2.2) -- (1,2.5);

\end{tikzpicture}
\hspace{1cm}
\begin{tikzpicture}[scale=.53]
\draw (0,0) -- (0,5);
\draw (0,5) -- (5,5);
\draw (5,5) -- (5,0);
\draw (5,0) -- (0,0);


\draw [line width=3](0,5) -- (0,1);
\draw [line width=3](0,1) -- (4,1);
\draw [line width=3](4,1) -- (4,0);
\draw [line width=3](4,0) -- (5,0);

\draw [dotted] (0,5) -- (5,0);

\node at (0.5,4.4) {{\large *}};
\node at (0.5,3.4) {{\large *}};
\node at (1.5,3.4) {{\large *}};
\node at (0.5,2.4) {{\large *}};
\node at (1.5,2.4) {{\large *}};
\node at (2.5,2.4) {{\large *}};
\node at (0.5,1.4) {{\large *}};
\node at (1.5,1.4) {{\large *}};
\node at (2.5,1.4) {{\large *}};
\node at (3.5,1.4) {{\large *}};
\node at (4.5,0.4) {{\large *}};

\node at (-.35,3) {4};
\node at (3.7,0.5) {1};


\end{tikzpicture}\]
\caption{$X_{(4,1)}$ and $X_{(4,1)} \cap \mathscr{L}$}
\label{lower}
\end{figure}
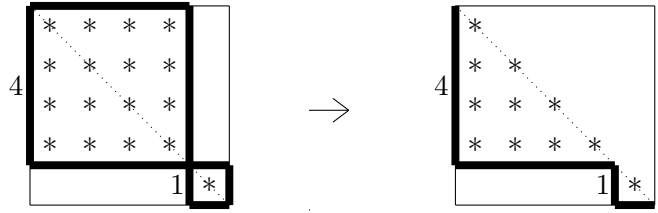

\begin{figure}[H]
\[\begin{tikzpicture}[scale=.53]
\draw (0,0) -- (0,5);
\draw (0,5) -- (5,5);
\draw (5,5) -- (5,0);
\draw (5,0) -- (0,0);

\draw [line width=3](0,5) -- (2,5);
\draw [line width=3](2,5) -- (2,3);
\draw [line width=3](2,3) -- (3,3);
\draw [line width=3](3,3) -- (3,2);
\draw [line width=3](3,2) -- (5,2);
\draw [line width=3](5,2) -- (5,0);

\draw [line width=3](0,5) -- (0,3);
\draw [line width=3](0,3) -- (2,3);
\draw [line width=3](2,3) -- (2,2);
\draw [line width=3](2,2) -- (3,2);
\draw [line width=3](3,2) -- (3,0);
\draw [line width=3](3,0) -- (5,0);

\draw [dotted] (0,5) -- (5,0);

\node at (0.5,4.4) {{\large *}};
\node at (1.5,4.4) {{\large *}};
\node at (0.5,3.4) {{\large *}};
\node at (1.5,3.4) {{\large *}};
\node at (2.5,2.4) {{\large *}};
\node at (3.5,1.4) {{\large *}};
\node at (4.5,1.4) {{\large *}};
\node at (3.5,0.4) {{\large *}};
\node at (4.5,0.4) {{\large *}};

\node at (1,5.4) {2};
\node at (2.5,3.4) {1};
\node at (4,2.4) {2};


\end{tikzpicture}
\hspace{1cm}
\begin{tikzpicture}[scale=.53]
\draw (0,0) -- (0,0);

\draw (.6,2.8) -- (1,2.5);
\draw (0,2.5) -- (1,2.5);
\draw (.6,2.2) -- (1,2.5);

\end{tikzpicture}
\hspace{1cm}
\begin{tikzpicture}[scale=.53]
\draw (0,0) -- (0,5);
\draw (0,5) -- (5,5);
\draw (5,5) -- (5,0);
\draw (5,0) -- (0,0);

\draw [line width=3](0,5) -- (2,5);
\draw [line width=3](2,5) -- (2,3);
\draw [line width=3](2,3) -- (3,3);
\draw [line width=3](3,3) -- (3,2);
\draw [line width=3](3,2) -- (5,2);
\draw [line width=3](5,2) -- (5,0);


\draw [dotted] (0,5) -- (5,0);

\node at (0.5,4.4) {{\large *}};
\node at (1.5,4.4) {{\large *}};
\node at (1.5,3.4) {{\large *}};
\node at (2.5,2.4) {{\large *}};
\node at (3.5,1.4) {{\large *}};
\node at (4.5,1.4) {{\large *}};
\node at (4.5,0.4) {{\large *}};

\node at (1,5.4) {2};
\node at (2.5,3.4) {1};
\node at (4,2.4) {2};


\end{tikzpicture}
\]
\caption{$X_{(2,1,2)}$ and $X_{(2,1,2)} \cap \mathscr{U}$}
\label{upper}
\end{figure}
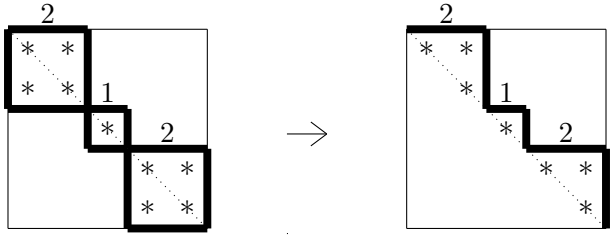
\end{example}


\begin{figure}[H]
\[\begin{tikzpicture}[scale=.53]
\draw (0,0) -- (0,5);
\draw (0,5) -- (5,5);
\draw (5,5) -- (5,0);
\draw (5,0) -- (0,0);

\draw [line width=3](0,5) -- (2,5);
\draw [line width=3](2,5) -- (2,3);
\draw [line width=3](2,3) -- (3,3);
\draw [line width=3](3,3) -- (3,2);
\draw [line width=3](3,2) -- (5,2);
\draw [line width=3](5,2) -- (5,0);

\draw [line width=3](0,5) -- (0,1);
\draw [line width=3](0,1) -- (4,1);
\draw [line width=3](4,1) -- (4,0);
\draw [line width=3](4,0) -- (5,0);

\draw [dotted] (0,5) -- (5,0);

\node at (0.5,4.4) {{\large *}};
\node at (1.5,4.4) {{\large *}};
\node at (0.5,3.4) {{\large *}};
\node at (1.5,3.4) {{\large *}};
\node at (0.5,2.4) {{\large *}};
\node at (1.5,2.4) {{\large *}};
\node at (2.5,2.4) {{\large *}};
\node at (0.5,1.4) {{\large *}};
\node at (1.5,1.4) {{\large *}};
\node at (2.5,1.4) {{\large *}};
\node at (3.5,1.4) {{\large *}};
\node at (4.5,1.4) {{\large *}};
\node at (4.5,0.4) {{\large *}};

\node at (-.35,3) {4};
\node at (3.7,0.5) {1};

\node at (1,5.4) {2};
\node at (2.5,3.4) {1};
\node at (4,2.4) {2};


\end{tikzpicture}
\]
\caption{
The defining matrix form of the seaweed $\mf{p}_5\frac{4|1}{2|1|2}$}
\label{seaweed}
\end{figure}
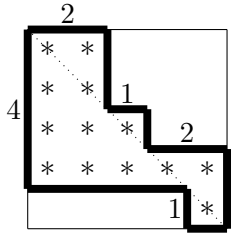

To each seaweed $\mathfrak{g}$,
we associate a planar graph, called a \textit{{meander}}. 
The general meander construction is evident from an illustrative example.

\begin{example} To construct the meander associated with the seaweed 
$\mf{p}_5\frac{4|1}{2|1|2}$ from Example \ref{construction}, we first place five vertices $v_1$ through $v_5$ in a horizontal line (we illustrate the vertices by their subscripts in bubbles as shown in Figure \ref{meander} (left)).  Next, create two partitions of the vertices by forming \textit{{top}} and {\textit{bottom blocks}} of vertices of size $4$, $1$ and $2$, $1$, $2$, respectively. These blocks are grouped according to the upper and lower vertical separators (see Figure \ref{meander} (middle)).  Now, for each block, add an edge from the first vertex of the block to the last vertex of the same block.  Repeat this edge addition  on the second vertex and the second to last vertex within the same block and so on within each block of both partitions.  Edge additions are between distinct vertices.  Top block edges are drawn concave down, and bottom block edges are drawn concave up. The completed meander $\mathcal{M}\left(\mf{p}_5\frac{4|1}{2|1|2}\right)$ is shown in Figure \ref{meander} (right). Ongoing, we will omit the separator lines in the finished meander.  

\end{example}

\begin{figure}[H]

\[\begin{tikzpicture}[scale=.67]

\vertex (11) at (1,2.5) {1};
\vertex (21) at (2,2.5) {2};
\vertex (31) at (3,2.5) {3};
\vertex (41) at (4,2.5) {4};
\vertex (51) at (5,2.5) {5};

\draw (2.5,1.5) -- (2.5,1.5);
\draw (2.5,3.5) -- (2.5,3.5);

;\end{tikzpicture}
\hspace{1.5cm}
\begin{tikzpicture}[scale=.67]

\vertex (11) at (1,2.5) {1};
\vertex (21) at (2,2.5) {2};
\vertex (31) at (3,2.5) {3};
\vertex (41) at (4,2.5) {4};
\vertex (51) at (5,2.5) {5};

\draw [line width=.75](4.5,2.75) -- (4.5,3.5);
\draw [line width=.75](2.5,2.25) -- (2.5,1.5);
\draw [line width=.75](3.5,2.25) -- (3.5,1.5);

;\end{tikzpicture}
\hspace{1.5cm}
\begin{tikzpicture}[scale=.67]

\vertex (1) at (1,2.5) {1};
\vertex (2) at (2,2.5) {2};
\vertex (3) at (3,2.5) {3};
\vertex (4) at (4,2.5) {4};
\vertex (5) at (5,2.5) {5};

\draw (1) to [bend left=50] (4);
\draw (2) to [bend left=50] (3);
\draw (1) to [bend right=50] (2);
\draw (4) to [bend right=50] (5);

\draw [line width=.75](4.5,2.75) -- (4.5,3.5);
\draw [line width=.75](2.5,2.25) -- (2.5,1.5);
\draw [line width=.75](3.5,2.25) -- (3.5,1.5);

;\end{tikzpicture}\]
\caption{Construction of the meander $\mathcal{M}\left(\mf{p}_5\frac{4|1}{2|1|2}\right)$}
\label{meander}
\end{figure}
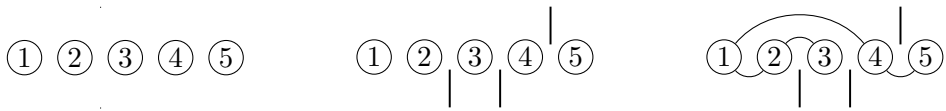

We can compute the index of the algebra by counting the number and type of connected components in the seaweed's meander via formula (\ref{first}) in the following beautiful combinatorial result.

\begin{theorem}[Dergachev and A. Kirillov \textbf{\cite{DK}}]\label{general linear}
If $\mathfrak{g}$ is a seaweed subalgebra of $\mathfrak{gl}(n)$, then 
\begin{eqnarray}\label{first}
\rm{ind}~\mathfrak{g}= 2C+P, 
\end{eqnarray}
where $C$ is the number of cycles and $P$ is the number of paths in $\mathcal{M}(\mathfrak{g})$. Note that isolated vertices are considered degenerate paths and contribute to $P$.
\end{theorem}

\begin{example}
By Theorem \ref{general linear}, the seaweed $\mf{p}_5\frac{4|1}{2|1|2}$ from Example \ref{construction} has index one since its meander consists of a single path.  See Figure \ref{meander} (right).  Note that it also follows from Theorem \ref{general linear} that there are no Frobenius seaweeds in $\mathfrak{gl}(n)$.
\end{example}

We hasten to add that while Theorem \ref{general linear} is combinatorially elegant, it is difficult to apply in practice.  To illustrate this, we ask the reader to look at the meander in Figure \ref{seaweed1} and try to quickly discern the index -- before looking at the colorized version in Figure \ref{seaweed2} where the paths and cycles have been identified.

\begin{figure}[H]
\begin{center}
\[\begin{tikzpicture}[scale=.53]

\draw (1,1) node[draw,circle,fill=black,minimum size=5pt,inner sep=0pt] (1) {};
\draw (2,1) node[draw,circle,fill=black,minimum size=5pt,inner sep=0pt] (2) {};
\draw (3,1) node[draw,circle,fill=black,minimum size=5pt,inner sep=0pt] (3) {};
\draw (4,1) node[draw,circle,fill=black,minimum size=5pt,inner sep=0pt] (4) {};
\draw (5,1) node[draw,circle,fill=black,minimum size=5pt,inner sep=0pt] (5) {};
\draw (6,1) node[draw,circle,fill=black,minimum size=5pt,inner sep=0pt] (6) {};
\draw (7,1) node[draw,circle,fill=black,minimum size=5pt,inner sep=0pt] (7) {};
\draw (8,1) node[draw,circle,fill=black,minimum size=5pt,inner sep=0pt] (8) {};
\draw (9,1) node[draw,circle,fill=black,minimum size=5pt,inner sep=0pt] (9) {};
\draw (10,1) node[draw,circle,fill=black,minimum size=5pt,inner sep=0pt] (10){};
\draw (11,1) node[draw,circle,fill=black,minimum size=5pt,inner sep=0pt] (11){};
\draw (12,1) node[draw,circle,fill=black,minimum size=5pt,inner sep=0pt] (12) {};
\draw (13,1) node[draw,circle,fill=black,minimum size=5pt,inner sep=0pt] (13) {};
\draw (14,1) node[draw,circle,fill=black,minimum size=5pt,inner sep=0pt] (14) {};
\draw (15,1) node[draw,circle,fill=black,minimum size=5pt,inner sep=0pt] (15){};
\draw (16,1) node[draw,circle,fill=black,minimum size=5pt,inner sep=0pt] (16){};
\draw (17,1) node[draw,circle,fill=black,minimum size=5pt,inner sep=0pt] (17) {};
\draw (18,1) node[draw,circle,fill=black,minimum size=5pt,inner sep=0pt] (18) {};
\draw (19,1) node[draw,circle,fill=black,minimum size=5pt,inner sep=0pt] (19) {};
\draw (20,1) node[draw,circle,fill=black,minimum size=5pt,inner sep=0pt] (20){};
\draw (21,1) node[draw,circle,fill=black,minimum size=5pt,inner sep=0pt] (21){};
\draw (22,1) node[draw,circle,fill=black,minimum size=5pt,inner sep=0pt] (22) {};
\draw (23,1) node[draw,circle,fill=black,minimum size=5pt,inner sep=0pt] (23) {};
\draw (24,1) node[draw,circle,fill=black,minimum size=5pt,inner sep=0pt] (24) {};
\draw (25,1) node[draw,circle,fill=black,minimum size=5pt,inner sep=0pt] (25){};
\draw (26,1) node[draw,circle,fill=black,minimum size=5pt,inner sep=0pt] (26){};
\path
(1) edge[bend left=50,color=black,line width=1.2pt] (5)
(2) edge[bend left=50,color=black,line width=1.2pt] (4)
(6) edge[bend left=50,color=black,line width=1.2pt] (12)
(7) edge[bend left=50,color=black,line width=1.2pt] (11)
(8) edge[bend left=50,color=black,line width=1.2pt] (10)
(13) edge[bend left=50,color=black,line width=1.2pt] (16)
(14) edge[bend left=50,color=black,line width=1.2pt] (15)
(17) edge[bend left=50,color=black,line width=1.2pt] (26)
(18) edge[bend left=50,color=black,line width=1.2pt] (25)
(19) edge[bend left=50,color=black,line width=1.2pt] (24)
(20) edge[bend left=50,color=black,line width=1.2pt] (23)
(21) edge[bend left=50,color=black,line width=1.2pt] (22)

(1) edge[bend right=50,color=black,line width=1.2pt] (8)
(2) edge[bend right=50,color=black,line width=1.2pt] (7)
(3) edge[bend right=50,color=black,line width=1.2pt] (6)
(4) edge[bend right=50,color=black,line width=1.2pt] (5)
(9) edge[bend right=50,color=black,line width=1.2pt] (14)
(10) edge[bend right=50,color=black,line width=1.2pt] (13)
(11) edge[bend right=50,color=black,line width=1.2pt] (12)
(15) edge[bend right=50,color=black,line width=1.2pt] (20)
(16) edge[bend right=50,color=black,line width=1.2pt] (19)
(17) edge[bend right=50,color=black,line width=1.2pt] (18)
(21) edge[bend right=50,color=black,line width=1.2pt] (26)
(22) edge[bend right=50,color=black,line width=1.2pt] (25)
(23) edge[bend right=50,color=black,line width=1.2pt] (24)

;\end{tikzpicture}\]
\caption{Meander for the seaweed $\mf{p}_{26}\; \frac{5 | 7 | 4 | 10}{8|6|6|6}$} 
\label{seaweed1}
\end{center}
\end{figure}
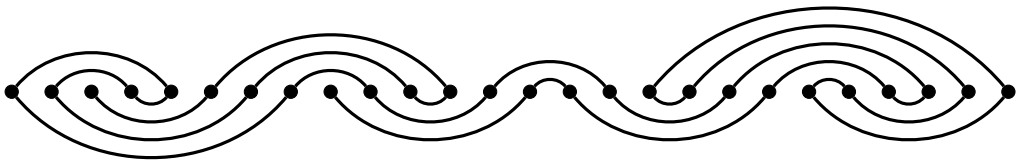

\begin{figure}[H]
\begin{center}
\[\begin{tikzpicture}[scale=.53]

\draw (1,1) node[draw,circle,fill=blue,minimum size=5pt,inner sep=0pt] (1) {};
\draw (2,1) node[draw,circle,fill=blue,minimum size=5pt,inner sep=0pt] (2) {};
\draw (3,1) node[draw,circle,fill=blue,minimum size=5pt,inner sep=0pt] (3) {};
\draw (4,1) node[draw,circle,fill=blue,minimum size=5pt,inner sep=0pt] (4) {};
\draw (5,1) node[draw,circle,fill=blue,minimum size=5pt,inner sep=0pt] (5) {};
\draw (6,1) node[draw,circle,fill=blue,minimum size=5pt,inner sep=0pt] (6) {};
\draw (7,1) node[draw,circle,fill=blue,minimum size=5pt,inner sep=0pt] (7) {};
\draw (8,1) node[draw,circle,fill=blue,minimum size=5pt,inner sep=0pt] (8) {};
\draw (9,1) node[draw,circle,fill=blue,minimum size=5pt,inner sep=0pt] (9) {};
\draw (10,1) node[draw,circle,fill=blue,minimum size=5pt,inner sep=0pt] (10){};
\draw (11,1) node[draw,circle,fill=blue,minimum size=5pt,inner sep=0pt] (11){};
\draw (12,1) node[draw,circle,fill=blue,minimum size=5pt,inner sep=0pt] (12) {};
\draw (13,1) node[draw,circle,fill=blue,minimum size=5pt,inner sep=0pt] (13) {};
\draw (14,1) node[draw,circle,fill=blue,minimum size=5pt,inner sep=0pt] (14) {};
\draw (15,1) node[draw,circle,fill=blue,minimum size=5pt,inner sep=0pt] (15){};
\draw (16,1) node[draw,circle,fill=blue,minimum size=5pt,inner sep=0pt] (16){};
\draw (17,1) node[draw,circle,fill=red,minimum size=5pt,inner sep=0pt] (17) {};
\draw (18,1) node[draw,circle,fill=red,minimum size=5pt,inner sep=0pt] (18) {};
\draw (19,1) node[draw,circle,fill=blue,minimum size=5pt,inner sep=0pt] (19) {};
\draw (20,1) node[draw,circle,fill=blue,minimum size=5pt,inner sep=0pt] (20){};
\draw (21,1) node[draw,circle,fill=red,minimum size=5pt,inner sep=0pt] (21){};
\draw (22,1) node[draw,circle,fill=red,minimum size=5pt,inner sep=0pt] (22) {};
\draw (23,1) node[draw,circle,fill=blue,minimum size=5pt,inner sep=0pt] (23) {};
\draw (24,1) node[draw,circle,fill=blue,minimum size=5pt,inner sep=0pt] (24) {};
\draw (25,1) node[draw,circle,fill=red,minimum size=5pt,inner sep=0pt] (25){};
\draw (26,1) node[draw,circle,fill=red,minimum size=5pt,inner sep=0pt] (26){};
\path
(1) edge[bend left=50,color=blue,line width=1.2pt] (5)
(2) edge[bend left=50,color=blue,line width=1.2pt] (4)
(6) edge[bend left=50,color=blue,line width=1.2pt] (12)
(7) edge[bend left=50,color=blue,line width=1.2pt] (11)
(8) edge[bend left=50,color=blue,line width=1.2pt] (10)
(13) edge[bend left=50,color=blue,line width=1.2pt] (16)
(14) edge[bend left=50,color=blue,line width=1.2pt] (15)
(17) edge[bend left=50,color=red,line width=1.2pt] (26)
(18) edge[bend left=50,color=red,line width=1.2pt] (25)
(19) edge[bend left=50,color=blue,line width=1.2pt] (24)
(20) edge[bend left=50,color=blue,line width=1.2pt] (23)
(21) edge[bend left=50,color=red,line width=1.2pt] (22)

(1) edge[bend right=50,color=blue,line width=1.2pt] (8)
(2) edge[bend right=50,color=blue,line width=1.2pt] (7)
(3) edge[bend right=50,color=blue,line width=1.2pt] (6)
(4) edge[bend right=50,color=blue,line width=1.2pt] (5)
(9) edge[bend right=50,color=blue,line width=1.2pt] (14)
(10) edge[bend right=50,color=blue,line width=1.2pt] (13)
(11) edge[bend right=50,color=blue,line width=1.2pt] (12)
(15) edge[bend right=50,color=blue,line width=1.2pt] (20)
(16) edge[bend right=50,color=blue,line width=1.2pt] (19)
(17) edge[bend right=50,color=red,line width=1.2pt] (18)
(21) edge[bend right=50,color=red,line width=1.2pt] (26)
(22) edge[bend right=50,color=red,line width=1.2pt] (25)
(23) edge[bend right=50,color=blue,line width=1.2pt] (24)
;\end{tikzpicture}\]
\caption{$\ind \mf{p}_{26}\; \frac{5 | 7 | 4 | 10}{8|6|6|6} = 3$} 
\label{seaweed2}
\end{center}
\end{figure}
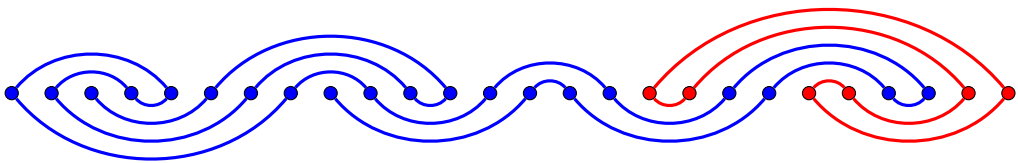
\section{Type-A seaweeds}\label{type A}
A type-A seaweed $\mathfrak{g}$ is a Lie subalgebra of $\mf{sl}(n)$ and is parametrized by two ordered compositions of $n$ as in Section \ref{two}; we write $\mathfrak{g}=\mathfrak{p}_n^\A \frac{a_1|\cdots|a_m}{b_1|\cdots|b_t}$. Note that the admissible locations for $\mathfrak{p}_n^\A \frac{a_1|\cdots|a_m}{b_1|\cdots|b_t}$ and  $\mathfrak{p}_n \frac{a_1|\cdots|a_m}{b_1|\cdots|b_t}$
are identical, so the two seaweeds have identical matrix forms and meanders.  However, the zero-trace condition for type-A seaweeds alters the index in formula (\ref{first}) slightly as follows.

\begin{theorem}\label{special linear}
If $\mathfrak{g}$ is a seaweed subalgebra of $\mathfrak{sl}(n)$, then 
\begin{eqnarray}\label{formula2}
\rm{ind}~\mathfrak{g}= 2C+P-1, 
\end{eqnarray}
where $C$ is the number of cycles and $P$ is the number of paths in $\mathcal{M}(\mathfrak{g})$.
\end{theorem}

\begin{example}
The seaweed $\mf{p}_5^\A\frac{4|1}{2|1|2}$ is Frobenius by Theorem \ref{special linear} since its meander consists of a single path.  See Figure \ref{meander} (right).

\end{example}  

Of course, when there is a large number of vertices in the meander, Theorem \ref{special linear} is just as nettlesome to apply as Theorem \ref{general linear}.  It would be advantageous to have an index formula for a type-A seaweed that relies on the compositions that define it.  Fortunately, these exist and can be derived from (\ref{formula2}).

\begin{theorem}[Coll et al. \textbf{\cite{Coll1}}, 2015]\label{thm:3parts}
The seaweed $\mathfrak{p}_n^\A\frac{a|b|c}{n}$ $\textup($or $\mathfrak{p}_n^\A\frac{a|b}{c|n-c}\textup)$ has index $\gcd(a+b,b+c)-1.$
\end{theorem}

Letting $b=0$ in Theorem \ref{thm:3parts} yields the following corollary, whose essential algebraic form was known to Elashvili much earlier.

\begin{theorem}[Elashvili \textbf{\cite{Elash}}, 1990]\label{thm:elashvili}
The seaweed $\mathfrak{p}_n^\A\frac{a|c}{n}$ has index $\gcd(a,c)-1.$
\end{theorem}

\begin{example} Applying  Theorem \ref{thm:elashvili} to $\mf{p}^\A_8\frac{3|5}{8}$ and $\mf{p}^\A_8\frac{4|4}{8}$ gives index values of 0 and 3, respectively. The meanders are illustrated in Figure \ref{seaweed4}, and we see that the index results are consistent with the formula in (\ref{formula2}).


\end{example}

\begin{figure}[H]
\[\begin{tikzpicture}[scale=.6]

\draw (1,1) node[draw,circle,fill=black,minimum size=5pt,inner sep=0pt] (1) {};
\draw (2,1) node[draw,circle,fill=black,minimum size=5pt,inner sep=0pt] (2) {};
\draw (3,1) node[draw,circle,fill=black,minimum size=5pt,inner sep=0pt] (3) {};
\draw (4,1) node[draw,circle,fill=black,minimum size=5pt,inner sep=0pt] (4) {};
\draw (5,1) node[draw,circle,fill=black,minimum size=5pt,inner sep=0pt] (5) {};
\draw (6,1) node[draw,circle,fill=black,minimum size=5pt,inner sep=0pt] (6) {};
\draw (7,1) node[draw,circle,fill=black,minimum size=5pt,inner sep=0pt] (7) {};
\draw (8,1) node[draw,circle,fill=black,minimum size=5pt,inner sep=0pt] (8) {};
\path
(1) edge[bend left=50,color=black,line width=1.2pt] (3)
(4) edge[bend left=50,color=black,line width=1.2pt] (8)
(5) edge[bend left=50,color=black,line width=1.2pt] (7)

(1) edge[bend right=50,color=black,line width=1.2pt] (8)
(2) edge[bend right=50,color=black,line width=1.2pt] (7)
(3) edge[bend right=50,color=black,line width=1.2pt] (6)
(4) edge[bend right=50,color=black,line width=1.2pt] (5)

;\end{tikzpicture}
\hspace{1.5cm}
\begin{tikzpicture}[scale=.55]

\draw (1,1) node[draw,circle,fill=blue,minimum size=5pt,inner sep=0pt] (1) {};
\draw (2,1) node[draw,circle,fill=red,minimum size=5pt,inner sep=0pt] (2) {};
\draw (3,1) node[draw,circle,fill=red,minimum size=5pt,inner sep=0pt] (3) {};
\draw (4,1) node[draw,circle,fill=blue,minimum size=5pt,inner sep=0pt] (4) {};
\draw (5,1) node[draw,circle,fill=blue,minimum size=5pt,inner sep=0pt] (5) {};
\draw (6,1) node[draw,circle,fill=red,minimum size=5pt,inner sep=0pt] (6) {};
\draw (7,1) node[draw,circle,fill=red,minimum size=5pt,inner sep=0pt] (7) {};
\draw (8,1) node[draw,circle,fill=blue,minimum size=5pt,inner sep=0pt] (8) {};
\path
(1) edge[bend left=50,color=blue,line width=1.2pt] (4)
(2) edge[bend left=50,color=red,line width=1.2pt] (3)
(5) edge[bend left=50,color=blue,line width=1.2pt] (8)
(6) edge[bend left=50,color=red,line width=1.2pt] (7)

(1) edge[bend right=50,color=blue,line width=1.2pt] (8)
(2) edge[bend right=50,color=red,line width=1.2pt] (7)
(3) edge[bend right=50,color=red,line width=1.2pt] (6)
(4) edge[bend right=50,color=blue,line width=1.2pt] (5)

;\end{tikzpicture}\]
\caption{
$\mathcal{M}\left(\mf{p}^\A_8\;\frac{3|5}{8}\right)$ and $\mathcal{M}\left(\mf{p}^\A_8\;\frac{4|4}{8}\right)$}
\label{seaweed4}
\end{figure}
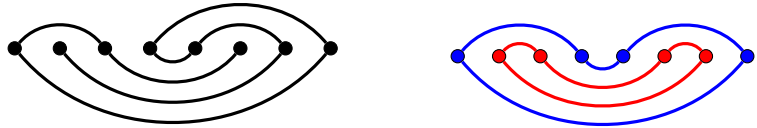

\noindent
The linear greatest common divisor formulas of Theorems \ref{thm:3parts} and \ref{thm:elashvili} are clearly handy for index calculations. But are there similar formulas when there are more than three parts in the top composition, or, equivalently, if there are more than four total parts in the defining compositions?  After extensive simulations,\footnote{Each meander can be contracted or ``wound down" to the empty meander through a sequence of graph-theoretic moves, each of which is uniquely determined by the structure of the meander at the time of the move application. The index of the associated seaweed can be tracked during this winding-down procedure, and the procedure can be reversed to ``wind up" meanders of any size and configuration. In this way, a large test-bed 
of seaweed algebras was constructed against which potential linear gcd index conditions were tested. The winding moves are a critically important 
part of meander technology and are documented in detail in a pair of elementary articles 
\textbf{\cite{Coll3, Coll2}}.} Coll et al. (\textbf{\cite{Sig}}, 2012) conjectured that no such general formula existed. This was settled in 2015 by 
Karnauhova and Liebscher, who used complexity arguments to verify the conjecture, and showed even more, per the following result.

\begin{theorem}[Karnauhova and Liebscher \textbf{\cite{Kar}}, 2015]\label{5 parts} 
Consider the seaweed
$\mathfrak{g}=\mathfrak{p}_n^\A\frac{a_1|\cdots|a_m}{n}$. If $m\geq 4$, then there do not exist homogeneous polynomials $f_1,f_2\in \Z[x_1,\dots ,x_m]$ of arbitrary degree such that the number of connected components of 
$\mathcal{M}(\mathfrak{g})$ is given by 
$\gcd(f_1(a_1,\dots ,a_m),f_2(a_1,\dots ,a_m))$.
\end{theorem}




\begin{remark} Karnauhova and Liebscher's result informs us that if we are looking for general linear, or polynomial, index formulas, then we must restrict the number of parts in the defining compositions to at most four.  Theorems \ref{thm:3parts} and \ref{thm:elashvili} deal with the four and three-part cases, respectively.  
We leave it as an easy exercise in the application of Theorem \ref{special linear} to establish an index formula for the two-part case as a function of the terms in the defining compositions.  
We will find in types B, C, and D that the part restriction is even more severe than in type A: we are limited to a maximum of three parts in the defining (partial) compositions. In these other classical types, we will only consider the maximal (three-part) case.    

\end{remark}

\section{Type-C seaweeds}\label{type C}

A type-C seaweed $\mathfrak{g}$ is a Lie subalgebra of $\mf{sp}(2n)$ and is
parametrized by two ordered \textit{partial} compositions of $n$ as follows.  Let 
$$\ul{a}=(a_1,a_2,\dots ,a_m)~~ and~~ \ul{b}=(b_1,b_2,\dots ,b_t) 
~with~ \displaystyle \sum_{i=1}^{m}a_i \leq n ~and~ 
\displaystyle \sum_{i=1}^{t}b_i \leq n.$$
We will assume, without loss of generality, that $\sum a_i \geq \sum b_i$.  
As with type-A seaweeds, type-C seaweeds have seaweed shape. Following the notation of Section \ref{two}, 
let $X_{ \ul{a} }$ be the vector space of block-diagonal matrices whose blocks now have sizes 
\begin{eqnarray}\label{typecblocks}
a_1\times a_1,\dots,a_m\times a_m, 2\left(n-\sum a_i\right) \times 2\left(n-\sum a_i\right), a_m \times a_m, \dots, a_1 \times a_1
\end{eqnarray}
and similarly for $X_{\ul{b}}$. Then, the underlying vector space for the type-C seaweed is a subspace of 
$M_{2n}(\CC)$ and is the span over $\CC$ of 
$$(X_{\ul{a}} \cap \mathscr{L}) ~\cup~ (X_{\ul{b}} \cap \mathscr{U}).$$
Following the notation for type A, we denote a type-C seaweed by $\mf{p}_{n}^\C(\ul{a} \dd \ul{b})$ or $\mathfrak{p}_n^\C \frac{a_1|\cdots|a_m}{b_1|\cdots|b_t}$. 
Using (\ref{typecblocks}), the reader will have no difficulty in constructing figures similar to those of Figures \ref{lower} and \ref{upper}.  

\begin{example}\label{ex:Cseaweed}

The seaweed $\mf{p}^\C_5\frac{1|4}{3}$ is a subalgebra of $\mathfrak{sp}(10)$, and the underlying vector space is illustrated in Figure \ref{Cseaweed}, where the $*$'s indicate admissible locations.  

\end{example}

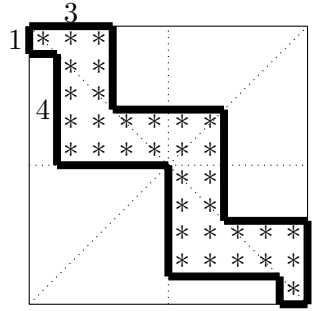
\begin{figure}[H]
\[\begin{tikzpicture}[scale=.37]
\draw (0,0) -- (0,10);
\draw (0,10) -- (10,10);
\draw (10,10) -- (10,0);
\draw (10,0) -- (0,0);

\draw [line width=3](0,10) -- (0,9);
\draw [line width=3](0,9) -- (1,9);
\draw [line width=3](1,9) -- (1,5);
\draw [line width=3](1,5) -- (5,5);
\draw [line width=3](5,5) -- (5,1);
\draw [line width=3](5,1) -- (9,1);
\draw [line width=3](9,1) -- (9,0);
\draw [line width=3](9,0) -- (10,0);

\draw [line width=3](0,10) -- (3,10);
\draw [line width=3](3,10) -- (3,7);
\draw [line width=3](3,7) -- (7,7);
\draw [line width=3](7,7) -- (7,3);
\draw [line width=3](7,3) -- (10,3);
\draw [line width=3](10,3) -- (10,0);

\draw [dotted] (0,0) -- (10,10);
\draw [dotted] (0,10) -- (10,0);
\draw [dotted] (5,0) -- (5,10);
\draw [dotted] (0,5) -- (10,5);

\node at (0.5,9.4) {{\large *}};
\node at (1.5,9.4) {{\large *}};
\node at (1.5,8.4) {{\large *}};
\node at (1.5,7.4) {{\large *}};
\node at (1.5,6.4) {{\large *}};
\node at (1.5,5.4) {{\large *}};
\node at (2.5,9.4) {{\large *}};
\node at (2.5,8.4) {{\large *}};
\node at (2.5,7.4) {{\large *}};
\node at (2.5,6.4) {{\large *}};
\node at (2.5,5.4) {{\large *}};
\node at (3.5,6.4) {{\large *}};
\node at (3.5,5.4) {{\large *}};
\node at (4.5,6.4) {{\large *}};
\node at (4.5,5.4) {{\large *}};
\node at (5.5,6.4) {{\large *}};
\node at (5.5,5.4) {{\large *}};
\node at (5.5,4.4) {{\large *}};
\node at (5.5,3.4) {{\large *}};
\node at (5.5,2.4) {{\large *}};
\node at (5.5,1.4) {{\large *}};
\node at (6.5,6.4) {{\large *}};
\node at (6.5,5.4) {{\large *}};
\node at (6.5,4.4) {{\large *}};
\node at (6.5,3.4) {{\large *}};
\node at (6.5,2.4) {{\large *}};
\node at (6.5,1.4) {{\large *}};
\node at (7.5,2.4) {{\large *}};
\node at (7.5,1.4) {{\large *}};
\node at (8.5,2.4) {{\large *}};
\node at (8.5,1.4) {{\large *}};
\node at (9.5,2.4) {{\large *}};
\node at (9.5,1.4) {{\large *}};
\node at (9.5,.4) {{\large *}};

\node at (-.5,9.5) {1};
\node at (.5,7) {4};

\node at (1.5,10.5) {3};


\end{tikzpicture}
\]
\caption{
The defining matrix form of the seaweed $\mf{p}^\C_5\frac{1|4}{3}$}
\label{Cseaweed}
\end{figure}

\subsection{Type-C meanders and a combinatorial index formula}
The construction of a type-C meander associated with $\mf{p}_n^\C(\ul{a} \dd \ul{b})$ begins the same way as in type A but is different in two ways.  Firstly, certain edges are omitted as follows. We insert an additional top and bottom separator to mark the ends of the top and bottom partial compositions.  We require that no top edge can contain vertices beyond the inserted top separator and similarly for the bottom.  Secondly, we 
designate a special subset of vertices $T_n^\C(\ul{a} \dd \ul{b})$
called the \textit{tail} of the meander as follows. Let  
 $r=\sum a_i$ and $s=\sum b_i$.  The tail is given by
$$T^\C_n(\ul{a}\dd\ul{b}) = \{v_{s+1},v_{s+2},\dots ,v_r\}.$$
Example \ref{separator} illustrates the meander construction and tail identification.


\begin{example}\label{separator} 
The meander of $\mf{p}^\C_5\frac{1|4}{3}$ from Example \ref{ex:Cseaweed} is illustrated in Figure \ref{Cmeander}, where the tail vertices are highlighted in yellow -- a convention we will follow ongoing.
Here, the tail $T^\C_5\left( (1,4) \dd (3) \right) =\{4,5\}$.  The additional vertical separators are illustrated in red.  Note that in type A, both compositions are full, so there is no type-A tail.
\end{example}
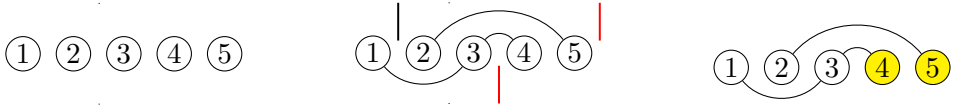
\begin{figure}[H]
\[\begin{tikzpicture}[scale=.67]
\vertex (11) at (1,2.5) {1};
\vertex (21) at (2,2.5) {2};
\vertex (31) at (3,2.5) {3};
\vertex (41) at (4,2.5) {4};
\vertex (51) at (5,2.5) {5};

\draw (2.5,1.5) -- (2.5,1.5);
\draw (2.5,3.5) -- (2.5,3.5);

;\end{tikzpicture}
\hspace{1.5cm}
\begin{tikzpicture}[scale=.67]

\vertex (11) at (1,2.5) {1};
\vertex (21) at (2,2.5) {2};
\vertex (31) at (3,2.5) {3};
\vertex (41) at (4,2.5) {4};
\vertex (51) at (5,2.5) {5};

\draw [line width=.75](1.5,2.75) -- (1.5,3.5);
\draw [color = red, line width=.75](3.5,2.25) -- (3.5,1.5);
\draw [color = red, line width=.75](5.5,2.75) -- (5.5,3.5);
\draw (2.5,1.5) -- (2.5,1.5);
\draw (2.5,3.5) -- (2.5,3.5);

\draw (31) to [bend left=50] (41);
\draw (21) to [bend left=50] (51);
\draw (11) to [bend right=50] (31);

;\end{tikzpicture}
\hspace{1.5cm}
\begin{tikzpicture}[scale=.67]

\vertex (1) at (1,0) {1};
\vertex (2) at (2,0) {2};
\vertex (3) at (3,0) {3};
\vertex[fill=yellow] (4) at (4,0) {4};
\vertex[fill=yellow] (5) at (5,0) {5};

\draw (3) to [bend left=50] (4);
\draw (2) to [bend left=50] (5);
\draw (1) to [bend right=50] (3);

\draw (2.5,1.5) -- (2.5,1.5);
\draw (2.5,3.5) -- (2.5,3.5);

;\end{tikzpicture}\]
\caption{Construction of the meander $\mathcal{M}\left(\mf{p}^\C_5\frac{1|4}{3}\right)$}
\label{Cmeander}
\end{figure}






We now have the following theorem, the type-C analogue of Theorems \ref{general linear} and \ref{special linear}. 

\begin{theorem}[Coll et al. \textbf{\cite{CHM}}, Theorem 4.5]\label{symplectic index}
If $\mathfrak{g} = \mathfrak{p}_n^\C(\ul{a}\dd\ul{b})$ is a seaweed subalgebra of $\mathfrak{sp}(2n)$, then 
\begin{eqnarray}\label{formula3}
\rm{ind}~\mathfrak{g}= 2C+\widetilde{P}, 
\end{eqnarray}
where $C$ is the number of cycles and $\widetilde{P}$ is the number of connected
components containing either zero or two vertices from $T_n^\C(\ul{a}\dd\ul{b})$.

\end{theorem}

\begin{example}
In $\mathcal{M}\left(\mf{p}^\C_5\frac{1|4}{3}\right)$, we find $C = 0$ and $\widetilde{P} = 0$, so by Theorem \ref{symplectic index}, the seaweed $\mf{p}^\C_5\frac{1|4}{3}$ is Frobenius. See Figure \ref{Cmeander} (right).  
\end{example} 

The tail allows us to completely classify Frobenius type-C seaweeds up to similarity.  The combinatorial formula in Theorem \ref{symplectic index} is zero when $C$ and $\widetilde{P}$ are both zero, i.e., when all components of the meander are paths with one endpoint in the tail.  We record this in the following corollary.

\begin{theorem}\label{TypeCForest}
A type-C seaweed is Frobenius if and only if its corresponding meander is a forest rooted in the tail.  
\end{theorem} 

\begin{example}\label{TypeCFrobeniusExample} The seaweed $\mf{p}_{14}^\C\frac{7 | 7}{11}$ is Frobenius by Theorem \ref{TypeCForest}. See Figure \ref{CFrobExample} for $\mathcal{M}\left(\mf{p}_{14}^\C\frac{7 | 7}{11}\right)$, where there are three paths in the associated meander. These are colored red, blue, and black.  Each path is rooted in one of the tail vertices. 

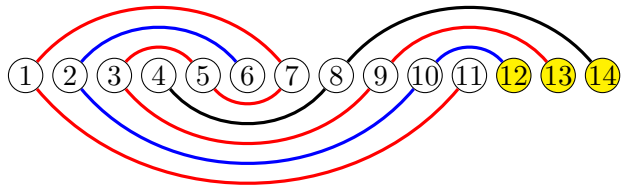
\begin{figure}[H]
\[\begin{tikzpicture}[scale=.59]
\vertex (1) at (1,0) {1};
\vertex (2) at (2,0) {2};
\vertex (3) at (3,0) {3};
\vertex (4) at (4,0) {4};
\vertex (5) at (5,0) {5};
\vertex (6) at (6,0) {6};
\vertex (7) at (7,0) {7};
\vertex (8) at (8,0) {8};
\vertex (9) at (9,0) {9};
\vertex (10) at (10,0) {10};
\vertex (11) at (11,0) {11};
\vertex[fill=yellow] (12) at (12,0) {12};
\vertex[fill=yellow] (13) at (13,0) {13};
\vertex[fill=yellow] (14) at (14,0) {14};

\path  
(1) edge[bend left=50, color=red, line width=1.2pt] (7)
(2) edge[bend left=50, color=blue, line width=1.2pt] (6)
(3) edge[bend left=50, color=red, line width=1.2pt] (5)
(8) edge[bend left=50, color=black, line width=1.2pt] (14)
(9) edge[bend left=50, color=red, line width=1.2pt] (13)
(10) edge[bend left=50, color=blue, line width=1.2pt] (12)
(1) edge[bend right=50, color=red, line width=1.2pt] (11)
(2) edge[bend right=50, color=blue, line width=1.2pt] (10)
(3) edge[bend right=50, color=red, line width=1.2pt] (9)
(4) edge[bend right=50, color=black, line width=1.2pt] (8)
(5) edge[bend right=50, color=red, line width=1.2pt] (7)
;\end{tikzpicture}\]
\caption{The meander $\mathcal{M}\left(\mf{p}_{14}^\C\frac{7 | 7}{11}\right)$ with components highlighted}
\label{CFrobExample} 
\end{figure}
\end{example}
\noindent 

\subsection{Type-C linear gcd index formulas}\label{CFormulas}
In the type-A case, Theorem \ref{5 parts} informs us that a general linear gcd index formula is not available when the total number of parts in a seaweed's defining compositions exceeds four. In type C, we can do no better.  In fact, we are limited to a total of three parts. To see this, let $\mathcal{M}_\C$ be the meander for $\mf{p}_{n}^\C \frac{a | b | c}{d}$, and let $\mathcal{M}_\A$ be the meander for $\mf{p}_{2n-d}^\A \frac{n-d | a | b | c}{2n-d}$. 
There are at most $\lf \frac{d}{2} \rf$ more cycles and paths in $\mathcal{M}_\A$ than there are in $\mathcal{M}_\C$, so the order of complexity associated with counting the cycles and paths in each of $\mathcal{M}_\C$ and $\mathcal{M}_\A$ is the same -- and equal to $\mathcal{O}(n)$.  Consequently, the index of $\mf{p}_{n}^\C \frac{a | b | c}{d}$ cannot be found using a gcd condition on homogeneous polynomials.  See Example \ref{C4}.

\begin{example}\label{C4}

Consider 
$\mf{p}_{n}^\C \frac{a | b | c}{d}$ with 
$n=15, a=5, b=4, c=6$, and $d=12$.  
See Figure \ref{DFourBlocks} for $\mathcal{M}_\C$ (blue) and $\mathcal{M}_\A$ (blue and red).  
By Theorem \ref{5 parts}, there are no homogeneous polynomials that can be used to count the number of cycles and paths in $\mathcal{M}_\A$.  
The blue meander in Figure \ref{DFourBlocks} is that of a four-part type-C seaweed which can be extended to the five-part blue-and-red type-A seaweed whose connected components cannot be counted by a polynomial gcd formula.
Intuitively, a type-C seaweed with four total parts corresponds to a five-part type-A seaweed using a construction similar to that used in Figure \ref{DFourBlocks}, and the order of complexity is the same.



\end{example}

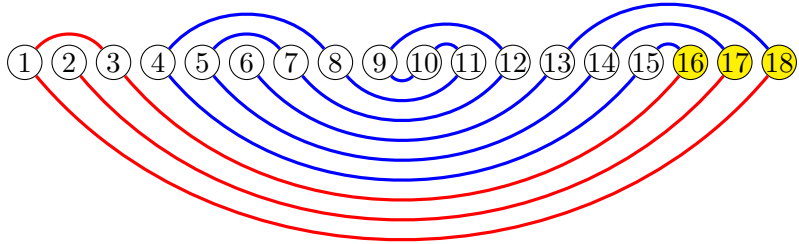
\begin{figure}[H]
\[\begin{tikzpicture}[scale=.59]

\vertex (1) at (1,0) {1};
\vertex (2) at (2,0) {2};
\vertex (3) at (3,0) {3};
\vertex (4) at (4,0) {4};
\vertex (5) at (5,0) {5};
\vertex (6) at (6,0) {6};
\vertex (7) at (7,0) {7};
\vertex (8) at (8,0) {8};
\vertex (9) at (9,0) {9};
\vertex (10) at (10,0) {10};
\vertex (11) at (11,0) {11};
\vertex (12) at (12,0) {12};
\vertex (13) at (13,0) {13};
\vertex (14) at (14,0) {14};
\vertex (15) at (15,0) {15};
\vertex[fill=yellow] (16) at (16,0) {16};
\vertex[fill=yellow] (17) at (17,0) {17};
\vertex[fill=yellow] (18) at (18,0) {18};

\draw[color=red, line width=1.2 pt] (1) to [bend left=50] (3);
\draw[color=blue, line width=1.2 pt] (4) to [bend left=50] (8);
\draw[color=blue, line width=1.2 pt] (5) to [bend left=50] (7);
\draw[color=blue, line width=1.2 pt] (9) to [bend left=50] (12);
\draw[color=blue, line width=1.2 pt] (10) to [bend left=50] (11);
\draw[color=blue, line width=1.2 pt] (13) to [bend left=50] (18);
\draw[color=blue, line width=1.2 pt] (14) to [bend left=50] (17);
\draw[color=blue, line width=1.2 pt] (15) to [bend left=50] (16);
\draw[color=red, line width=1.2 pt] (1) to [bend right=50] (18);
\draw[color=red, line width=1.2 pt] (2) to [bend right=50] (17);
\draw[color=red, line width=1.2 pt] (3) to [bend right=50] (16);
\draw[color=blue, line width=1.2 pt] (4) to [bend right=50] (15);
\draw[color=blue, line width=1.2 pt] (5) to [bend right=50] (14);
\draw[color=blue, line width=1.2 pt] (6) to [bend right=50] (13);
\draw[color=blue, line width=1.2 pt] (7) to [bend right=50] (12);
\draw[color=blue, line width=1.2 pt] (8) to [bend right=50] (11);
\draw[color=blue, line width=1.2 pt] (9) to [bend right=50] (10);

;\end{tikzpicture}\]
\caption{The meanders $\mathcal{M}_\C$ and $\mathcal{M}_\A$} 
\label{DFourBlocks}
\end{figure}

The following theorem is the type-C analogue of Theorem \ref{thm:elashvili}.  Note that the top composition is full, while the bottom trivial composition is not.

\begin{theorem}[Coll et al. \textbf{\cite{CHM}}, Theorem 5.2]\label{3 parts thm1}
Let $a+b=n$. If $c=n-1$ or $c=n-2$, then
\begin{eqnarray}\label{formula5}
\ind \mf{p}_n^\C\ \frac{a | b}{c}=\gcd(a+b,b+c)-1.
\end{eqnarray}
\end{theorem}

The following result is a companion to Theorem \ref{3 parts thm1}. Here, the trivial composition is full.

\begin{theorem}[Coll et al. \textbf{\cite{CHM}}, Theorem 5.5]\label{3 parts thm3}  Let $a$, $b$, and $n$ be positive integers.

\begin{enumerate}[\textup(i\textup)]

\item If $a+b=n-1$ then $\ind \mf{p}_n^\C\dfrac{n}{a|b}=\gcd(a+b,b+1)-1$.

\item If $a+b=n-2$ then $\ind \mf{p}_n^\C\dfrac{n}{a|b}=\gcd(a+b,b+2)-1$.

\end{enumerate}
\end{theorem}

Theorems \ref{3 parts thm1}, \ref{3 parts thm3}, and \ref{symplectic index} allow us to quickly identify Frobenius seaweeds directly from their defining partial compositions.

\begin{theorem}[Coll et al. \textbf{\cite{CHM}}, Theorem 5.6]\label{3 parts thm2}
If $a+b=n$, then $\ind \mf{p}_n^\C\frac{a | b}{c}=0$ if and only if one of the following conditions hold:
\begin{enumerate}[\textup(i\textup)]
\item $c=n-1$ and $\gcd(a+b,b+c)=1,$
\item $c=n-2$ and $\gcd(a+b,b+c)=1,$
\item $c=n-3$, the integers $a,b,$ and $c$ are all odd, and $\gcd(a+b,b+c)=2$. 
\end{enumerate}
\end{theorem}

\begin{example} We display Frobenius seaweeds in each of the three cases of Theorem
\ref{3 parts thm2}.
$$
(i)~~\mf{p}_8^\C\frac{4 | 4}{7}~~~~~
(ii)~~\mf{p}_8^\C\frac{5 | 3}{6}~~~~~
(iii)~~\mf{p}_8^\C\frac{7 | 1}{5}
$$
\end{example}







We close this section with an exercise that addresses the two-part case in type C. 

\bigskip
\noindent
\textbf{Exercise.} 
Show that if $a=b$, then $\ind \mf{p}_n^\C\frac{~a~}{b}=n$.
Otherwise,
\[\ind \mf{p}_n^\C\frac{~a~}{b}=
\begin{cases}
n-a+\lf\frac{a-b}{2}\rf, & \text{ if }n\text{ is even;}\\
n-a+\lf\frac{a-b-1}{2}\rf, & \text{ if }n\text{ is odd}.
\end{cases}\]



\section{Type-B seaweeds}\label{type B}

A type-B seaweed $\mathfrak{g}$ is a Lie subalgebra of $\mf{so}(2n+1)$ and is
parametrized by two ordered partial compositions of $n$ exactly as in type C.  Type-B seaweeds have seaweed shape, and the analogue of (\ref{typecblocks}) is given by
\begin{eqnarray}\label{blocks}
a_1\times a_1,\dots,a_m\times a_m, 2\left(n-\sum a_i\right)+1 \times 2\left(n-\sum a_i\right)+1, a_m \times a_m, \dots, a_1 \times a_1,
\end{eqnarray}

\noindent
Following the notation for type A and type C, we denote a type-B seaweed by $\mf{p}_{n}^\B(\ul{a} \dd \ul{b})$ or $\mathfrak{p}_n^\B \frac{a_1|\cdots|a_m}{b_1|\cdots|b_t}$. 
Using (\ref{blocks}), the reader will have no difficulty in constructing figures similar to those of 
Figures \ref{lower} and \ref{upper}.  
From here, the type-B meander construction, the combinatorial index formula in (\ref{formula3}), and all other index formulas of Section \ref{CFormulas} apply verbatim to type-B seaweeds.  We simply provide an example.

\begin{example}\label{ex:Bseaweed}
The seaweed $\mf{p}^\B_5\frac{3|2}{4}$ is a subalgebra of $\mathfrak{so}(11)$, and the underlying vector space is illustrated in Figure \ref{Bseaweed} (left).  
Its associated meander $\mathcal{M}\left(\mf{p}^\B_5\frac{3|2}{4}\right)$ is illustrated in Figure \ref{Bseaweed} (right).  
Note that the seaweed is Frobenius, as is $\mf{p}^\C_5\frac{3|2}{4}$.  
\end{example}
\begin{figure}[H]
\[\begin{tikzpicture}[scale=.37]
\draw (0,0) -- (0,11);
\draw (0,11) -- (11,11);
\draw (11,11) -- (11,0);
\draw (11,0) -- (0,0);

\draw [line width=3](0,11) -- (4,11);
\draw [line width=3](4,11) -- (4,7);
\draw [line width=3](4,7) -- (7,7);
\draw [line width=3](7,7) -- (7,4);
\draw [line width=3](7,4) -- (11,4);
\draw [line width=3](11,4) -- (11,0);
\draw [line width=3](0,11) -- (0,8);
\draw [line width=3](0,8) -- (3,8);
\draw [line width=3](3,8) -- (3,6);
\draw [line width=3](3,6) -- (5,6);
\draw [line width=3](5,6) -- (5,5);
\draw [line width=3](5,5) -- (6,5);
\draw [line width=3](6,5) -- (6,3);
\draw [line width=3](6,3) -- (8,3);
\draw [line width=3](8,3) -- (8,0);
\draw [line width=3](8,0) -- (11,0);

\draw [dotted] (0,0) -- (11,11);

\node at (0.5,10.4) {{\large *}};
\node at (1.5,10.4) {{\large *}};
\node at (2.5,10.4) {{\large *}};
\node at (3.5,10.4) {{\large *}};

\node at (0.5,9.4) {{\large *}};
\node at (1.5,9.4) {{\large *}};
\node at (2.5,9.4) {{\large *}};
\node at (3.5,9.4) {{\large *}};

\node at (0.5,8.4) {{\large *}};
\node at (1.5,8.4) {{\large *}};
\node at (2.5,8.4) {{\large *}};
\node at (3.5,8.4) {{\large *}};

\node at (3.5,7.4) {{\large *}};

\node at (3.5,6.4) {{\large *}};
\node at (4.5,6.4) {{\large *}};
\node at (5.5,6.4) {{\large *}};
\node at (6.5,6.6) {{$0$}};

\node at (5.5,5.6) {{$0$}};
\node at (6.5,5.4) {{\large *}};

\node at (6.5,4.4) {{\large *}};

\node at (6.5,3.4) {{\large *}};
\node at (7.5,3.4) {{\large *}};
\node at (8.5,3.4) {{\large *}};
\node at (9.5,3.4) {{\large *}};
\node at (10.5,3.4) {{\large *}};

\node at (8.5,2.4) {{\large *}};
\node at (9.5,2.4) {{\large *}};
\node at (10.5,2.4) {{\large *}};

\node at (8.5,1.4) {{\large *}};
\node at (9.5,1.4) {{\large *}};
\node at (10.5,1.4) {{\large *}};

\node at (8.5,.4) {{\large *}};
\node at (9.5,.4) {{\large *}};
\node at (10.5,.4) {{\large *}};

\node at (-.5,9.5) {3};
\node at (2.5,7) {2};

\node at (2,11.5) {4};

\end{tikzpicture}
\hspace{1.5cm}
\begin{tikzpicture}[scale=.67]

\vertex (1) at (1,0) {1};
\vertex (2) at (2,0) {2};
\vertex (3) at (3,0) {3};
\vertex (4) at (4,0) {4};
\vertex[fill=yellow] (5) at (5,0) {5};

\draw (3,-3) -- (3,-3);
\draw (3,3) -- (3,3);

\draw (1) to [bend left=50] (3);
\draw (4) to [bend left=50] (5);
\draw (1) to [bend right=50] (4);
\draw (2) to [bend right=50] (3);

;\end{tikzpicture}\]
\caption{
The defining matrix form of the seaweed $\mf{p}^\B_5\frac{3|2}{4}$ and its associated meander $\mathcal{M}\left(\mf{p}^\B_5\frac{3|2}{4}\right)$}
\label{Bseaweed}
\end{figure}
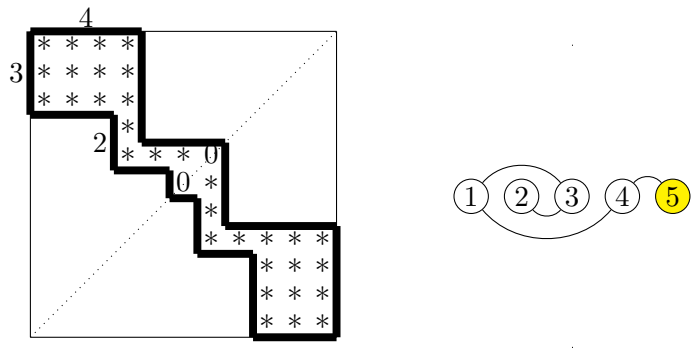


\section{Type-D seaweeds}\label{type D}

A type-D seaweed $\mathfrak{g}$ is a Lie subalgebra of $\mf{so}(2n)$, and the type-D seaweeds that we consider here (see the Epilogue) are parametrized by two ordered partial  compositions of $n$ exactly as in types C and B. The underlying matrix form of these type-D seaweeds has seaweed shape, and the block decomposition is identical to that used in type C.  (Of course, the algebras are different because of the different symmetry conditions imposed by Theorem \ref{classification}.)  However, unlike in type B, the index theory of type-D seaweeds is radically different from that of type C.

\begin{example}\label{ex:dseaweed}

The seaweed $\mf{p}^\D_5\frac{1|4}{2}$ is a subalgebra of $\mathfrak{so}(10)$, and the underlying matrix form is illustrated in Figure \ref{Dseaweed}.  

\end{example}

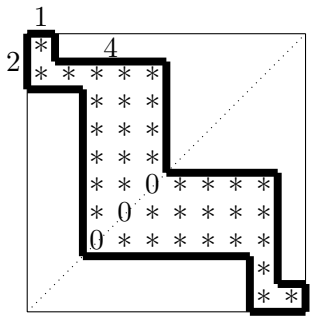
\begin{figure}[H]
\[\begin{tikzpicture}[scale=.37]
\draw (0,0) -- (0,10);
\draw (0,10) -- (10,10);
\draw (10,10) -- (10,0);
\draw (10,0) -- (0,0);

\draw [line width=3](0,10) -- (1,10);
\draw [line width=3](1,10) -- (1,9);
\draw [line width=3](1,9) -- (5,9);
\draw [line width=3](5,9) -- (5,5);
\draw [line width=3](5,5) -- (9,5);
\draw [line width=3](9,5) -- (9,1);
\draw [line width=3](9,1) -- (10,1);
\draw [line width=3](10,1) -- (10,0);
\draw [line width=3](0,10) -- (0,8);
\draw [line width=3](0,8) -- (2,8);
\draw [line width=3](2,8) -- (2,2);
\draw [line width=3](2,2) -- (8,2);
\draw [line width=3](8,2) -- (8,0);
\draw [line width=3](8,0) -- (10,0);

\draw [dotted] (0,0) -- (10,10);

\node at (0.5,9.4) {{\large *}};

\node at (0.5,8.4) {{\large *}};
\node at (1.5,8.4) {{\large *}};
\node at (2.5,8.4) {{\large *}};
\node at (3.5,8.4) {{\large *}};
\node at (4.5,8.4) {{\large *}};

\node at (2.5,7.4) {{\large *}};
\node at (3.5,7.4) {{\large *}};
\node at (4.5,7.4) {{\large *}};

\node at (2.5,6.4) {{\large *}};
\node at (3.5,6.4) {{\large *}};
\node at (4.5,6.4) {{\large *}};

\node at (2.5,5.4) {{\large *}};
\node at (3.5,5.4) {{\large *}};
\node at (4.5,5.4) {{\large *}};

\node at (2.5,4.4) {{\large *}};
\node at (3.5,4.4) {{\large *}};
\node at (4.5,4.6) {{$0$}};
\node at (5.5,4.4) {{\large *}};
\node at (6.5,4.4) {{\large *}};
\node at (7.5,4.4) {{\large *}};
\node at (8.5,4.4) {{\large *}};

\node at (2.5,3.4) {{\large *}};
\node at (3.5,3.6) {{$0$}};
\node at (4.5,3.4) {{\large *}};
\node at (5.5,3.4) {{\large *}};
\node at (6.5,3.4) {{\large *}};
\node at (7.5,3.4) {{\large *}};
\node at (8.5,3.4) {{\large *}};

\node at (2.5,2.6) {{$0$}};
\node at (3.5,2.4) {{\large *}};
\node at (4.5,2.4) {{\large *}};
\node at (5.5,2.4) {{\large *}};
\node at (6.5,2.4) {{\large *}};
\node at (7.5,2.4) {{\large *}};
\node at (8.5,2.4) {{\large *}};

\node at (8.5,1.4) {{\large *}};

\node at (8.5,.4) {{\large *}};
\node at (9.5,.4) {{\large *}};

\node at (-.5,9) {2};

\node at (.5,10.6) {1};
\node at (3,9.5) {4};

\end{tikzpicture}\]
\caption{
The defining matrix form of the seaweed $\mf{p}^\D_5\frac{1|4}{2}$}
\label{Dseaweed}
\end{figure}

\subsection{Type-D meanders}\label{dmeanders}

A type-D meander is formed exactly as a type-C meander, and as in the type-C case, we will find it helpful to define a distinguished set of vertices called the \textit{tail} of the meander. However, the type-D tail can take several ``configurations".  The following critical definition sets the notation. 

\begin{definition}\label{tailconfig}
Consider the seaweed $\mf{p}_n^\D(\ul{a} \dd \ul{b})$.  Let $t = \sum a_i - \sum b_j$.  We define the type-$\D$ tail of $\mf{p}_n^\D(\ul{a} \dd \ul{b})$ to be
\begin{eqnarray}\label{typeDparts}
T_n^\D(\ul{a}\dd\ul{b})=
\begin{cases}
T_n^\C(\ul{a}\dd\ul{b}),
& \text{ if } t \text{ is even;} \\
T_n^\C(\ul{a}\dd\ul{b}) + v_{1+\sum a_i},
& \text{ if } t\text{ is odd and } \sum a_i < n; \\
T_n^\C(\ul{a}\dd\ul{b}) - v_n,
& \text{ if } t\text{ is odd and } \sum a_i = n. \\
\end{cases}
\end{eqnarray}







\end{definition}

\begin{example}  See Figure \ref{Dseaweed2} for an example of a type-D meander.  The tail vertices are shaded yellow. Note that the meander construction is exactly the same as in type C -- but the tail identification may be different.
Here, $T^\D_5\left( (1,4) \dd (3) \right) =\{3,4\}$, whereas $T^\C_5\left( (1,4) \dd (3) \right) =\{3, 4,5\}$.

\end{example}

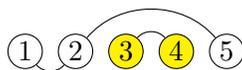
\begin{figure}[H]
$$\begin{tikzpicture}[scale=.67]
\vertex (1) at (1,0) {1};
\vertex (2) at (2,0) {2};
\vertex[fill=yellow] (3) at (3,0) {3};
\vertex[fill=yellow] (4) at (4,0) {4};
\vertex (5) at (5,0) {5};

\draw (2) to [bend left=50] (5);
\draw (3) to [bend left=50] (4);
\draw (1) to [bend right=50] (2);
\end{tikzpicture}$$
\caption{
The meander $\mathcal{M}\left(\mf{p}^\D_5\frac{1|4}{2}\right)$}
\label{Dseaweed2}
\end{figure}

The following is the type-D analogue of the combinatorial index formulas in Theorems \ref{general linear}, \ref{special linear}, and \ref{symplectic index}.  

\begin{theorem}\label{typeD}

If $\mathfrak{g} = \mathfrak{p}_n^\D(\ul{a}\dd\ul{b})$ is a subalgebra of $\mathfrak{so}(2n)$, then 
\begin{eqnarray}\label{formula4}
\rm{ind}~\mathfrak{g}= 2C+\widetilde{P}, 
\end{eqnarray}
where $C$ is the number of cycles and $\widetilde{P}$ is the number of connected
components containing either zero or two vertices from $T_n^\D(\ul{a}\dd\ul{b})$.

\end{theorem}

\begin{example}
In $\mathcal{M}\left(\mf{p}^\D_5\frac{1|4}{2}\right)$, we find $C = 0$ and $\widetilde{P} = 2$, so by Theorem \ref{typeD}, the index of $\mf{p}^\D_5\frac{1|4}{2}$ is $2$.
See Figure \ref{Dseaweed2}.  
\end{example}

With the established definition of the type-D tail, we can now concisely state the type-D analogue to the visual for type-C Frobenius seaweeds given in Theorem \ref{TypeCForest}.

\begin{theorem}\label{DForest}
A type-D seaweed is Frobenius if and only if its corresponding meander is a forest rooted in the tail.
\end{theorem}

\noindent
\textbf{Author's comment.} The reader may be curious as to how the authors came upon the definition of the type-D tail.  The idea was to define the tail in such a way so that Theorem \ref{typeD} could be stated exactly as that for type-C.  We knew the right theorem, but needed to find the right definition!  By so prescribing the tail, we were trophied with the analogous type-D visual given in Theorem \ref{DForest}, that is, a type-D Frobenius seaweed is one for which the meander is a forest rooted in the tail.

Since Section \ref{type D} is, of necessity, quite long, we preface our ongoing work with the following heuristic. As with type C (and B), the index of a seaweed is determined using the tail of the seaweed's meander.  The three type-D tail \textit{configurations}, which do not appear in type C, set out in Definition \ref{tailconfig} must be treated separately.  
The analysis for the first two configurations follows quickly from Theorem \ref{typeD} and some type-C results from Section \ref{type C}  (see Theorems \ref{configI} and \ref{configII}).
The analysis for the third configuration, however, is not so straightforward and is the subject of Section \ref{lastD}.   In this configuration, denoted $\mf{p}_n^\D\left(\frac{a | b}{c}, \textrm{III}\right)$,  we find that how the tail affects the index is determined by the relationships between $b$, $c$, and $n$.  The most interesting case is when is when $b > n-c$  (see Case 3 in Section \ref{lastD}).  We show that in order for $\mf{p}_n^\D\left(\frac{a | b}{c}, \textrm{III}\right)$, $b > n-c$ to be Frobenius, the tail must be of size two or four (Theorem \ref{tailsize}).  With the tail size thus limited, we 
consider the meander's \textit{delta}, which is an integer 
associated with the meander (see the CONSTRUCTION below).
Finally, using the delta, we give necessary and sufficient conditions for the seaweed to be Frobenius:  see Theorem \ref{HomotopyTypeH1} (for a tail of size two) and 
Theorem \ref{HomotopyTypeH11} (for a tail of size four).

\bigskip
Returning to the article narrative, we say that the tail, $T_n^\D(\ul{a}\dd\ul{b})$, has \textit{configuration} I, II, or III according to the three cases in (\ref{typeDparts}).  To ease notation, we will denote, for example, a seaweed
$\mf{p}_n^\D(\ul{a} \dd \ul{b})$ with \textit{tail configuration} III as $\mf{p}_n^\D\left((\ul{a} \dd \ul{b}), \textrm{III}\right)$, etc.
Again, when the compositions $\ul{a}$ and $\ul{b}$ are explicit, we will find it convenient to use the fractional notation $\mf{p}_n^\D \left( \frac{a_1 | ... | a_m}{b_1 | ... | b_t}, \textrm{III}\right)$. See Example \ref{3amigoes}.

\begin{example}\label{3amigoes}
In Figure \ref{tail}, we provide illustrative examples of the three cases in (\ref{typeDparts}).  The tail vertices are shaded yellow.
These seaweeds have index one, two, and one, respectively.  
\end{example}
\begin{figure}[H]\label{illustrate}
\begin{center}
\begin{tabular}{ccc}
$\begin{tikzpicture}[scale=.57, baseline=(current bounding box.center)]

\vertex (1) at (1,0) {1};
\vertex (2) at (2,0) {2};
\vertex (3) at (3,0) {3};
\vertex (4) at (4,0) {4};
\vertex[fill=yellow] (5) at (5,0) {5};
\vertex[fill=yellow] (6) at (6,0) {6};
\vertex[fill=yellow] (7) at (7,0) {7};
\vertex[fill=yellow] (8) at (8,0) {8};

\draw (1) to [bend left=50] (3);
\draw (4) to [bend left=50] (8);
\draw (5) to [bend left=50] (7);
\draw (1) to [bend right=50] (4);
\draw (2) to [bend right=50] (3);

;\end{tikzpicture}$
&
$\begin{tikzpicture}[scale=.57, baseline=(current bounding box.center)]

\vertex (1) at (1,0) {1};
\vertex (2) at (2,0) {2};
\vertex (3) at (3,0) {3};
\vertex (4) at (4,0) {4};
\vertex (5) at (5,0) {5};
\vertex (6) at (6,0) {6};
\vertex[fill=yellow] (7) at (7,0) {7};
\vertex[fill=yellow] (8) at (8,0) {8};
\vertex (9) at (9,0) {9};

\draw (1) to [bend left=50] (4);
\draw (2) to [bend left=50] (3);
\draw (5) to [bend left=50] (7);
\draw (1) to [bend right=50] (3);
\draw (4) to [bend right=50] (6);

;\end{tikzpicture}$
&
$\begin{tikzpicture}[scale=.57, baseline=(current bounding box.center)]

\vertex (1) at (1,0) {1};
\vertex (2) at (2,0) {2};
\vertex (3) at (3,0) {3};
\vertex (4) at (4,0) {4};
\vertex (5) at (5,0) {5};
\vertex (6) at (6,0) {6};
\vertex[fill=yellow] (7) at (7,0) {7};
\vertex[fill=yellow] (8) at (8,0) {8};
\vertex (9) at (9,0) {9};

\draw (1) to [bend left=50] (4);
\draw (2) to [bend left=50] (3);
\draw (5) to [bend left=50] (7);
\draw (8) to [bend left=50] (9);
\draw (1) to [bend right=50] (2);
\draw (3) to [bend right=50] (5);

;\end{tikzpicture}$
\end{tabular}
\caption{The meanders for the seaweeds $\mf{p}_{8}^\D \left(
\frac{3 | 5}{4}, \textrm{I}\right)$,
$\mf{p}_{9}^\D \left( \frac{4 | 3}{3 | 3}, \textrm{II}\right)$, and 
$\mf{p}_{9}^\D \left( \frac{4 | 3 | 2}{2 | 3 | 1}, \textrm{III}\right)$, respectively}
\label{tail}
\end{center}
\end{figure}
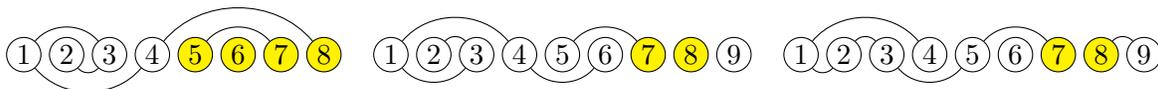

The following corollary reduces index computation for all type-D seaweeds with a tail of configuration I to results established in Section \ref{CFormulas}.  

\begin{theorem}\label{configI}
The index of $\mf{p}_n^\D\left((\ul{a} \dd \ul{b}), \I \right)$ equals the index of $\mf{p}_n^\C (\ul{a} \dd \ul{b})$.  
\end{theorem}


\begin{remark}\label{IRemark}
As a corollary of Theorem \ref{configI}, the type-C index formulas in Theorem \ref{3 parts thm1} with $c=n-2$ and Theorem \ref{3 parts thm2} case \textup($ii$\textup) apply to type-D seaweeds as well.  
\end{remark}



Proceeding as in types A, C, and B we now look into finding linear gcd index formulas.


\subsection{ Type-D index formulas}
With the established formula in Theorem \ref{typeD}, the complexity argument given in the preamble to
Example \ref{C4} identically applies to type-D seaweeds of the form
$\mf{p}_{n}^\D \frac{a | b | c}{d}$.  
We are left with the three-block case, 
$\mf{p}_{n}^\D \frac{a | b}{c}$.  
First, the configuration I case
$\mf{p}_n^\D\left(\frac{a | b}{c}, \I\right)$
is handled by Theorem \ref{configI} and Remark \ref{IRemark}.  
We are left with the consideration of configurations II and III.  To handle configuration II, we have the following direct corollary of Theorems \ref{3 parts thm1} and \ref{typeD}.  

\begin{theorem}\label{configII}
The index of $\mf{g} = \mf{p}_n^\D\left(\frac{a | b}{c}, \II\right)$ equals $n-(a+b+1)$ plus the index of
$\mf{p}_{a+b}^\C \frac{a | b}{c}$
\end{theorem}

\begin{proof}
In $\mathcal{M}(\mf{g})$, vertex $v_{a+b+1}$ is a path with one end in the tail, so it does not contribute to the index, and vertices $v_{a+b+2}, ..., v_{n}$ are paths with zero ends in the tail, each of which contributes one to the index.  
\end{proof}

Having dispensed with configurations I and II in Theorems \ref{configI} and \ref{configII}, respectively,  we are left to consider seaweeds of the form $\mf{p}_n^\D\left(\frac{a | b}{c}, \textrm{III}\right)$.  In such seaweeds, the component containing $v_n$ can contribute to the index differently depending on how $b$ and $n-c$ are related.  See Figure \ref{CaseThreeCases} for examples with $b = n-c$, $b < n-c$, and $b > n-c$, respectively.  

\begin{center}
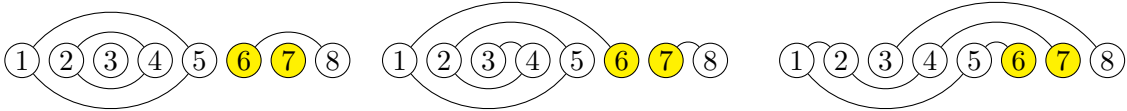
\begin{figure}[H]
\begin{tabular}{ccc}
$~~\begin{tikzpicture}[scale=.59, baseline=(current bounding box.center)]

\vertex (1) at (1,0) {1};
\vertex (2) at (2,0) {2};
\vertex (3) at (3,0) {3};
\vertex (4) at (4,0) {4};
\vertex (5) at (5,0) {5};
\vertex[fill=yellow] (6) at (6,0) {6};
\vertex[fill=yellow] (7) at (7,0) {7};
\vertex (8) at (8,0) {8};

\draw (3,-2) -- (3,-2);
\draw (3,2) -- (3,2);

\draw (1) to [bend left=50] (5);
\draw (2) to [bend left=50] (4);
\draw (6) to [bend left=50] (8);
\draw (1) to [bend right=50] (5);
\draw (2) to [bend right=50] (4);

;\end{tikzpicture}$
&
$\begin{tikzpicture}[scale=.59, baseline=(current bounding box.center)]

\vertex (1) at (1,0) {1};
\vertex (2) at (2,0) {2};
\vertex (3) at (3,0) {3};
\vertex (4) at (4,0) {4};
\vertex (5) at (5,0) {5};
\vertex[fill=yellow] (6) at (6,0) {6};
\vertex[fill=yellow] (7) at (7,0) {7};
\vertex (8) at (8,0) {8};

\draw (3,-2) -- (3,-2);
\draw (3,2) -- (3,2);

\draw (1) to [bend left=50] (6);
\draw (2) to [bend left=50] (5);
\draw (3) to [bend left=50] (4);
\draw (7) to [bend left=50] (8);
\draw (1) to [bend right=50] (5);
\draw (2) to [bend right=50] (4);

;\end{tikzpicture}$
&
$~~\begin{tikzpicture}[scale=.59, baseline=(current bounding box.center)]

\vertex (1) at (1,0) {1};
\vertex (2) at (2,0) {2};
\vertex (3) at (3,0) {3};
\vertex (4) at (4,0) {4};
\vertex (5) at (5,0) {5};
\vertex[fill=yellow] (6) at (6,0) {6};
\vertex[fill=yellow] (7) at (7,0) {7};
\vertex (8) at (8,0) {8};

\draw (3,-2) -- (3,-2);
\draw (3,2) -- (3,2);

\draw (1) to [bend left=50] (2);
\draw (3) to [bend left=50] (8);
\draw (4) to [bend left=50] (7);
\draw (5) to [bend left=50] (6);
\draw (1) to [bend right=50] (5);
\draw (2) to [bend right=50] (4);

;\end{tikzpicture}$
\end{tabular}
\caption{The meanders for the seaweeds $\mf{p}_{8}^\D \left( \frac{5 | 3}{5}, \textrm{III}\right)$, $\mf{p}_{8}^\D \left( \frac{6 | 2}{5}, \textrm{III}\right)$, and $\mf{p}_{8}^\D \left( \frac{2 | 6}{5}, \textrm{III}\right)$, respectively.}
\label{CaseThreeCases}
\end{figure}
\end{center}
\subsubsection{Configuration III seaweeds $\mf{p}_n^\D\left(\frac{a | b}{c}, \textrm{III}\right)$ }\label{lastD}

The analysis of this configuration breaks into three cases, illustrated by the examples in Figure \ref{CaseThreeCases}.  The third case is the most interesting.  We find that in this last case, if the seaweed is to be Frobenius, then the tail can only be of size two or four.

\bigskip
\noindent
\textbf{Case 1: } $b = n-c$ 

The  seaweed in Figure \ref{CaseThreeCases} (left) cannot be Frobenius since the components on
the first $a$ vertices are separated from the tail.  Moreover, we have the following general index formula.  

\begin{theorem}
If $b = n-c$, then 

\[\ind \mf{p}_n^\D\left(\frac{a | b}{c}, \III\right)=
\begin{cases}
a, & \text{ if }b = 1;\\
a + \lf \displaystyle \frac{b-3}{2} \rf, & \text{ if }b \geq 3.
\end{cases}\]
\end{theorem}


\noindent
\textbf{Case 2: } $b < n-c$

The seaweed in Figure \ref{CaseThreeCases} (middle) is Frobenius by Theorem~\ref{DForest}. Note that the subgraph on vertices $v_1$ through $v_6$ yields a Frobenius type-C meander.  In general, when $v_a$ and $v_{a+1}$ are tail vertices, the meander can be separated into two parts: one on the first $a$ vertices, and the other on the last $b$ vertices, with no arc connecting the two subgraphs. We find that for such a seaweed to be Frobenius, it must have $b=2$ or $3$, a result which holds for general seaweeds of this form.  

\begin{theorem}\label{Attaching Lemma}
If $\mf{p}_n^\D\left( \frac{a_1 | ... | a_m}{b_1 | ... | b_r}, \III\right)$ is Frobenius and $a_m < n - \sum b_i$, then $a_m = 2$ or $3$.  
\end{theorem}




As a corollary of 
Theorems \ref{3 parts thm2} and \ref{Attaching Lemma}, we can classify which seaweeds in Case 2 are Frobenius.  

\begin{theorem}\label{TypeDThreeBlock}
The seaweed $\mf{p}_n^\D\left(\frac{a | b}{c}, \III\right)$ with $b < n - c$ is Frobenius if and only if one of the following conditions holds:
\begin{enumerate}[\textup(i\textup)]
\item $b=2$ and $c=n-3$,
\item $b=3$, $c=n-5$, and $n$ is odd.
\end{enumerate}
\end{theorem} 

\noindent
\textbf{Case 3: } $b > n-c$

The seaweed in Figure \ref{CaseThreeCases} (right)
is an example of this case.  However,
without appeal to Theorem \ref{typeD},
we are not yet in a position to determine whether the seaweed is, or is not, 
Frobenius.  We first show that if $b > n-c$, then the tail, $T$, of a 
Frobenius seaweed $\mf{p}_n^\D\left(\frac{a | b}{c}, \textrm{III}\right)$ must have limited size.


\begin{theorem} \label{tailsize} If $\mf{p}_n^\D\left(\frac{a | b}{c}, \III\right)$ is Frobenius, then  
$c=n-3$ or $c=n-5$.  In particular, $|T| = 2$ or $4$.  
\end{theorem}

\proof
Suppose, for a contradiction, that $c \geq n-7$.  Then $\mf{p}_{n-1}^\C\frac{a | 1 | b-2}{c}$ is Frobenius by Theorem \ref{symplectic index}.  The seaweed $\mf{p}_{n-1}^\C\frac{a | 1 | b-2}{c}$ must have exactly $n-1-c \geq 6$ odd integers among its parts -- a contradiction. 
When $c=n-3$, the tail of $\mf{p}_n^\D\left(\frac{a | b}{c}, \III\right)$ is given by $T=\{v_{n-2}, v_{n-1}\}$.  When $c=n-5$, the tail of $\mf{p}_n^\D\left(\frac{a | b}{c}, \III\right)$
is given by $T=\{ v_{n-4},v_{n-3}, v_{n-2}, v_{n-1} \}$.
\qed 

Drilling down into Case 3, we now examine different tail sizes.

\bigskip
\begin{tcolorbox}[breakable, enhanced]
\begin{center}
\textbf{The tail of size two,  $|T|=2$}
\end{center}
\end{tcolorbox}

If $\mathfrak{p}_n^\D\left(\frac{a|b}{c},\I\I\I\right)$ is a Frobenius seaweed whose meander has tail of size two, then, by Theorem \ref{DForest}, its meander $\mathcal{M}\left(\mathfrak{p}_n^\D\left(\frac{a|b}{c},\I\I\I\right)\right)$ consists of exactly two paths. This can happen in two ways: either a path will connect $v_{n-2}$ and $v_n$, or a path will connect $v_{n-1}$ and $v_n.$ See Figures \ref{DFrobeniusH3} and \ref{DFrobeniusH1}, respectively.



\begin{figure}[H]
\[\begin{tikzpicture}[scale=.59]

\vertex (1) at (1,0) {1};
\vertex (2) at (2,0) {2};
\vertex (3) at (3,0) {3};
\vertex (4) at (4,0) {4};
\vertex (5) at (5,0) {5};
\vertex(6) at (6,0) {6};
\vertex[fill=yellow] (7) at (7,0) {7};
\vertex[fill=yellow]  (8) at (8,0) {8};
\vertex (9) at (9,0) {9};

\draw (1) to [bend left=50] (3);
\draw (4) to [bend left=50] (9);
\draw (5) to [bend left=50] (8);
\draw (6) to [bend left=50] (7);
\draw (1) to [bend right=50] (6);
\draw (2) to [bend right=50] (5);
\draw (3) to [bend right=50] (4);

;\end{tikzpicture}\]
\caption{In the meander of the seaweed $\mf{p}_{9}^\D \left( \frac{3 | 6}{6}, \III\right)$, one path connects $v_7$ and $v_9$.} 
\label{DFrobeniusH3}
\end{figure}
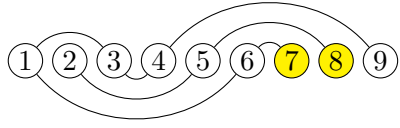

\begin{figure}[H]
\[\begin{tikzpicture}[scale=.59]

\vertex (1) at (1,0) {1};
\vertex (2) at (2,0) {2};
\vertex (3) at (3,0) {3};
\vertex (4) at (4,0) {4};
\vertex (5) at (5,0) {5};
\vertex (6) at (6,0) {6};
\vertex (7) at (7,0) {7};
\vertex[fill=yellow] (8) at (8,0) {8};
\vertex[fill=yellow] (9) at (9,0) {9};
\vertex (10) at (10,0) {10};

\draw (1) to [bend left=50] (4);
\draw (2) to [bend left=50] (3);
\draw (5) to [bend left=50] (10);
\draw (6) to [bend left=50] (9);
\draw (7) to [bend left=50] (8);
\draw (1) to [bend right=50] (7);
\draw (2) to [bend right=50] (6);
\draw (3) to [bend right=50] (5);

;\end{tikzpicture}\]

\caption{In the meander of the seaweed $\mf{p}_{10}^\D \left( \frac{4 | 6}{7}, \textrm{III}\right)$, one path connects $v_9$ and $v_{10}$.}
\label{DFrobeniusH1}
\end{figure}

If one path connects $v_{n-2}$ and $v_{n}$, then we have the following theorem.  







\begin{theorem}\label{HomotopyTypeH3}
If the meander of $\mf{p}_n^\D\left(\frac{a | b}{c}, \III\right)$ has a path connecting $v_{n-2}$ and $v_n$, then the seaweed is Frobenius if and only if $\gcd(a+b,b+c)=3$.  
\end{theorem}

On the other hand, if the meander of $\mf{p}_n^\D\left(\frac{a | b}{c}, \III\right)$ has a path connecting $v_{n-1}$ and $v_n$, then $\gcd(a+b,b+c)=1$.  However, this condition is not enough to guarantee $\mf{p}_n^\D\left(\frac{a | b}{c}, \III\right)$ is Frobenius.  See Example \ref{extension}.
\begin{example}\label{extension}
The seaweed $\mf{p}_{10}^\D \left( \frac{6 | 4}{7}, \textrm{III}\right)$ has $\gcd(6+4,4+7)=1$ but is not Frobenius (see Figure \ref{NotFrobeniusH1} (left)).  Note that by the addition of  a single (dotted) bottom edge, we get the meander of the 
Frobenius seaweed  $\mf{p}_{10}^\A \frac{6 | 4}{7| 3}$ (see Figure \ref{NotFrobeniusH1} (right)).  Since type-A meanders have no tail, we have eliminated the yellow-colored vertices. 
\end{example}

\begin{figure}[H]
\begin{center}
\begin{tabular}{cc}
$\begin{tikzpicture}[scale=.59]

\vertex (1) at (1,0) {1};
\vertex (2) at (2,0) {2};
\vertex (3) at (3,0) {3};
\vertex (4) at (4,0) {4};
\vertex (5) at (5,0) {5};
\vertex (6) at (6,0) {6};
\vertex (7) at (7,0) {7};
\vertex[fill=yellow] (8) at (8,0) {8};
\vertex[fill=yellow] (9) at (9,0) {9};
\vertex (10) at (10,0) {10};

\draw (1) to [bend left=50] (6);
\draw (2) to [bend left=50] (5);
\draw (3) to [bend left=50] (4);
\draw (7) to [bend left=50] (10);
\draw (8) to [bend left=50] (9);
\draw (1) to [bend right=50] (7);
\draw (2) to [bend right=50] (6);
\draw (3) to [bend right=50] (5);
;\end{tikzpicture}
\hspace{.25cm}$
&
$
\hspace{.25cm}
\begin{tikzpicture}[scale=.59]

\vertex (1) at (1,0) {1};
\vertex (2) at (2,0) {2};
\vertex (3) at (3,0) {3};
\vertex (4) at (4,0) {4};
\vertex (5) at (5,0) {5};
\vertex (6) at (6,0) {6};
\vertex (7) at (7,0) {7};
\vertex (8) at (8,0) {8};
\vertex (9) at (9,0) {9};
\vertex (10) at (10,0) {10};

\draw (1) to [bend left=50] (6);
\draw (2) to [bend left=50] (5);
\draw (3) to [bend left=50] (4);
\draw (7) to [bend left=50] (10);
\draw (8) to [bend left=50] (9);
\draw (1) to [bend right=50] (7);
\draw (2) to [bend right=50] (6);
\draw (3) to [bend right=50] (5);
\draw (8)[dotted] to [bend right=50] (10);
;\end{tikzpicture}$
\end{tabular}
\end{center}
\caption{The meanders for $\mf{p}_{10}^\D \left( \frac{6 | 4}{7}, \textrm{III}\right)$ and $\mf{p}_{10}^\A \frac{6 | 4}{7| 3}$.}
\label{NotFrobeniusH1}
\end{figure}
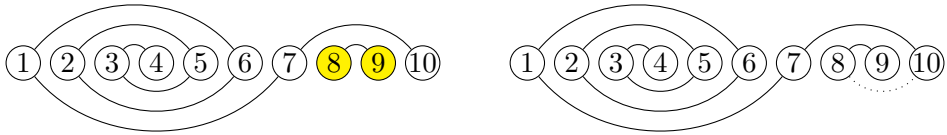

In particular, given $\mf{p}_n^\D\left(\frac{a | b}{c}, \III\right)$, if $\gcd(a+b,b+c)=1$, then one path connects $v_{n-2}$ and $v_n$, or one path connects $v_{n-2}$ and $v_{n-1}$.   
Comparing the meander in Figure \ref{DFrobeniusH1} to the meander in Figure \ref{NotFrobeniusH1} (left), we notice that for such $\mf{p}_{n}^\D \left( \frac{a | b}{c}, \textrm{III}\right)$ to be Frobenius, we seek a condition that will guarantee that vertices $v_{n-2}$ and $v_{n-1}$ are in different components.  To find this condition, we start with the following construction.

\newpage

\begin{tcolorbox}[breakable, enhanced]

\begin{center}
CONSTRUCTION
\end{center}
A Frobenius type-A seaweed $\mathfrak{g}$ has an associated meander $\mathcal{M}(\mathfrak{g})$ consisting of a single path with endpoint vertices $v_i$ and $v_f$.  In $\mathcal{M}(\mathfrak{g})$, we append self-loops to these vertices. Depending on $\mathcal{M(\mathfrak{g})}$, the self-loops may be top edges or bottom edges.  Let $M(\mathfrak{g})$ be the new planar graph defined by $\mathcal{M}(\mathfrak{g})$ together with these self-loops.  Now define two maps on the vertex set $V$ of $M(\mathfrak{g})$, called \textit{top} and \textit{bottom}, denoted by $t$ and $b$, respectively, and define them as follows.  Let $v\in V$; define $t(v)$ to be the vertex in $M(\mathfrak{g})$ connected to $v$ by a top edge in $M(\mathfrak{g})$ and $b(v)$ to be the vertex in $M(\mathfrak{g})$ connected to $v$ by a bottom edge in $M(\mathfrak{g})$. Now consider the composition map $t\circ b$ and its iterates. Keeping track of the subscripts of the vertices as each vertex is traversed produces an $n$-cycle in the obvious way.  
(Ongoing, and when convenient, we will use the subscript of a vertex to represent the vertex.)  We continue our construction in the context of Example \ref{extension} and consider the Frobenius seaweed $\mf{g}=\mf{p}_{10}^\A  \frac{6 | 4}{7 | 3}$.  The planar graph $M(\mf{g})$ is illustrated in Figure \ref{singledelta}. 

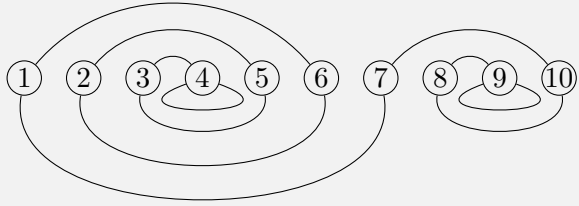
\begin{figure}[H]
\[\begin{tikzpicture}[scale=.79, baseline=(current bounding box.center)]

\vertex (1) at (1,-5) {1};
\vertex (2) at (2,-5) {2};
\vertex (3) at (3,-5) {3};
\vertex (4) at (4,-5) {4};
\Loop[dist=1.7cm,dir=SO,style={in=-20,out=-160}](4)
\vertex (5) at (5,-5) {5};
\vertex (6) at (6,-5) {6};
\vertex (7) at (7,-5) {7};
\vertex (8) at (8,-5) {8};
\vertex (9) at (9,-5) {9};
\Loop[dist=1.7cm,dir=SO,style={in=-20,out=-160}](9)
\vertex (10) at (10,-5) {10};

\draw (1) to [bend left=50] (6);
\draw (2) to [bend left=50] (5);
\draw (3) to [bend left=50] (4);
\draw (7) to [bend left=50] (10);
\draw (8) to [bend left=50] (9);
\draw (3) to [bend right=100] (5);
\draw (2) to [bend right=100] (6);
\draw (1) to [bend right=100] (7);
\draw (8) to [bend right=100] (10);

\end{tikzpicture}\]
\caption{$M(\mf{g})$ for $\mf{g}=\mf{p}_{10}^\A \; \frac{6 | 4}{7 | 3}$}
\label{singledelta}
\end{figure}

Using the notation above, $i=4$ and $f=9$. Now,
$t(b(4))=3 $, $t(b(3))=2$, $t(b(2))=1$, and so on. This produces the permutation cycle 
$\sigma=(4~3~2~1~10~9~8~7~6~5)$. 
The differences (mod 10) between consecutive elements of $\sigma$ form a multiset denoted $\Delta_\sigma$. Notice that in this example, all the elements of $\Delta_\sigma$
are the same.  Here this difference is $9$.
This ``single difference'' is not an isolated occurrence, rather it is a characteristic property of Frobenius seaweeds $\mf{p}_{n}^\A \; \frac{a | b}{c | d}$.  We call the value of this single difference $\Delta$, the \textit{delta} of the seaweed. See Theorem \ref{thm:mod}.

\end{tcolorbox}



\begin{theorem}\label{thm:mod}
If $\mf{p}_{n}^\A \; \frac{a | b}{c | d}$ is Frobenius, then $\Delta \equiv a+d ~(\bmod ~n)$.  
\end{theorem}

\begin{proof}
Let $\Delta(i)=\big(t(b(i))-i\big) (\bmod ~n)$, where $1 \leq i \leq n$.  We show $\Delta(i)\equiv a+d ~(\bmod ~n)$ for every $i$, and we then define $\Delta = \Delta(i)$.  
Suppose $1 \leq i \leq a$ and, without loss of generality, suppose $c>a$.  Since $1 \leq i \leq c$, we have that $b(i) = c+1-i$.  Then 
\[b(i) = c+1-i \leq c+1 - 1 = c \quad \quad \text{and} \quad \quad b(i) = c+1-i \geq c+1 - a.\]
Therefore,
$c+1-a \leq b(i) \leq c$,
so $b(i)$ is in block $c$, that is, in the first $c$ vertices.  Then
$b(i)+t(b(i))\equiv a+1~(\bmod ~n).$ 
Subtracting $b(i)$, we get that 
$t(b(i)) \equiv a-c+i ~(\bmod ~n).$
Since $c+d = n$, we have $\Delta(i)= t(b(i)) - i \equiv a+d ~(\bmod ~n).$

Next suppose $1+a \leq i \leq c$.  Since $1+a \leq i \leq c$, we have that $b(i) = c+1-i$.  Then 
\[b(i) = c+1-i \leq c+1 - (a+1) = c-a \quad \quad \text{and} \quad \quad b(i) = c+1-i \geq c+1 - c = 1.\]
It follows that
$1 \leq b(i) \leq c-a$, 
so $b(i)$ is in block $a$, or, in the case that $i = a+1$ and $v_i$ has a bottom edge that is a self-loop, we have $b(i)=i$. 
Either way, 
$b(i)+t(b(i))\equiv a+1~(\bmod ~n)$.
Subtracting $b(i)$, we find 
$t(b(i)) \equiv a-c+i ~(\bmod ~n)$.
Since $c+d = n$, we have
$\Delta(i)=t(b(i)) - i \equiv a+d ~(\bmod ~n)$.

Finally, suppose $1+c \leq i \leq a+b$, then $b(i) = 2c+d+1-i$.  Then 
\[b(i) = 2c+d+1-i \leq 2c+d+1-(c+1) = c+d ~~ \text{and} ~~ b(i) = 2c+d+1-i \geq 2c+d+1- (a+b) = c+1.\]
Therefore,
$c+1 \leq b(i) \leq c+d$,
so $b(i)$ is in block $d$, then
$b(i)+t(b(i)) = a+1+a+b \equiv a+1~(\bmod ~n)$.  
Subtracting $b(i)$, we find 
$t(b(i)) \equiv a-2c-d+i ~(\bmod ~n) \equiv a-c+i ~(\bmod ~n)$.  
Since $c+d = n$, we have
$\Delta(i)=t(b(i)) - i \equiv a+d ~(\bmod ~n)$.  
\end{proof}

\noindent
\textbf{Exercises.} The reader may wonder if Theorem \ref{thm:mod} is true for all Frobenius type-A seaweeds.  It is not.  However, what is true may be equally interesting.  Consider, for example, the Frobenius seaweed $\mf{g}=\mf{p}_8^\A \frac{1 |2 |5}{8}$. The permutation cycle is $\tau=(1~4~7~3~6~2~5~8)$ and 
the ordered multiset of differences thus obtained is $\Delta_\tau=\{3, 3, 4, 3, 4, 3, 3, 1\}$.  The distinct values appearing in $\Delta_\tau$ are $\Delta_1 = 1$, $\Delta_2 = 4$, and $\Delta_3 = 3$.  The cardinality of each $\Delta_i$
is the number of times each appears in $\Delta_\tau$.  So, $|\Delta_1|=1$, $|\Delta_2|=2$, and $|\Delta_3|=5$.  
Notice a pattern?  The block sizes in the top composition give these cardinalites! Now prove that a similar pattern holds for 
a Frobenius seaweed $\mathfrak{p}_n^\A\frac{a_1|\cdots|a_m}{n}.$  Must all the $\Delta_i$'s be distinct? 
Can you generalize the result to when the bottom composition is non-trivial?

\bigskip



We are now in a position to distinguish between the Frobenius and non-Frobenius cases described by Figures \ref{DFrobeniusH1} and \ref{NotFrobeniusH1} (left), respectively.  We find that coupling the necessary greatest common divisor condition with a congruence relation will classify certain families of Frobenius seaweeds.
Using the notation above, let 
$$
\xi(n,\Delta) = \displaystyle\frac{\Delta^{\varphi (n) - 1}}{n} - \lf \frac{\Delta^{\varphi (n) - 1}}{n}\rf,
$$
where $\varphi(n)$ is the Euler totient function\footnote{The Euler totient function $\varphi (n)$ counts the number of integers less than and relatively prime to $n$. } defined by
$\varphi(n) = n\displaystyle\prod_{p | n}\left( 1-\frac{1}{p} \right)$, and the product is over distinct prime numbers dividing $n$.

\begin{theorem}\label{HomotopyTypeH1}
If $\mf{g}=\mf{p}_n^\D\left(\frac{a | b}{c}, \III\right)$ is such that $\gcd(a+b,b+c) = 1$, then $\mathfrak{g}$ is Frobenius precisely when 
$$0 < \xi(n,\Delta)< 0.5.$$
\end{theorem}


\proof 

Let $\sigma$ be the permutation cycle for $\mf{p}_n^\A\frac{a | b}{c | 3}$.  Since $\Delta$ generates $\sigma$, there are distinct $k_1, k_2 \in (0,n)$ with 
\begin{eqnarray*}\label{tailofsize2proof}
n-1 +  k_1 \Delta &\equiv & n-2 ~(\bmod ~n), \\
n-1 + k_2 \Delta  &\equiv & 0 ~(\bmod ~n).
\end{eqnarray*}
If $k_1 > k_2$, then the path in the meander for $\mf{p}_n^\D\left(\frac{a | b}{c}, \III\right)$ containing $v_{n-1}$ also contains $v_n$.  In particular, if $k_1 > k_2$, then $v_{n-2}$ and $v_{n-1}$ are on different components in $\mf{p}_n^\D\left(\frac{a | b}{c}, \III\right)$, which, as noted earlier, is the second condition necessary for $\mf{p}_n^\D\left(\frac{a | b}{c}, \III\right)$ to be Frobenius.  These equations simplify to 
\begin{eqnarray}
k_1 \Delta  &\equiv & -1 ~(\bmod ~n), \\
k_2 \Delta  &\equiv & 1 ~(\bmod ~n). 
\end{eqnarray} 
\noindent
Multiplying equations (3) and (4) by $\Delta^{\varphi (n) - 1}$, and applying Euler's Totient theorem, yields the following system: 
\begin{eqnarray*}
k_1 &\equiv & -\Delta^{\varphi (n) - 1} ~(\bmod ~n), \\
k_2 &\equiv &~~ \Delta^{\varphi (n) - 1} ~(\bmod ~n). 
\end{eqnarray*} 
\noindent
If $k_1>k_2$, then 
$$-\Delta^{\varphi (n) - 1} ~(\bmod ~n) > \Delta^{\varphi (n) - 1} ~(\bmod ~n) ~\Longleftrightarrow~ 0 \leq \Delta^{\varphi (n) - 1} ~(\bmod ~n) \leq \frac{n}{2}.$$  But $k_i \neq 0$ for $i = 1,2$, so $\Delta^{\varphi (n) - 1} ~(\bmod ~n) \neq 0$.  Moreover, $k_1 \neq k_2$, so $\Delta^{\varphi (n) - 1} ~(\bmod ~n) \neq  \frac{n}{2}$.  The result follows.  \qed

Armed with Theorem \ref{HomotopyTypeH1}, we can now distinguish between the seaweeds in Figure \ref{DFrobeniusH1} and Figure \ref{NotFrobeniusH1} (left). Only one of these is Frobenius. See Example \ref{distinguish}.

\begin{example} \label{distinguish}
Note that the seaweeds of Figure \ref{DFrobeniusH1} and Figure \ref{NotFrobeniusH1} (left) both satisfy the gcd condition of Theorem \ref{HomotopyTypeH1}.

\begin{itemize}
\item The seaweed $\mf{p}_{10}^\D \left( \frac{4 | 6}{7}, \textrm{III}\right)$ is Frobenius:  $\Delta = 7$, $\varphi(10)=4$, and $\xi(10,7)= 0.3 <0.5 $.
\item
The seaweed $\mf{p}_{10}^\D \left( \frac{6 | 4}{7}, \textrm{III}\right)$ is not Frobenius: $\Delta = 9$, $\varphi (10) = 4$, and $\xi(10,9)= 0.9 > 0.5$.  
\end{itemize}
\end{example}


\begin{remark}
The above CONSTRUCTION and Theorem \ref{thm:mod} will also be useful when we analyze the tail of size four case, which we address next.
\end{remark}

\bigskip
\begin{tcolorbox}[breakable, enhanced]
\begin{center}
\textbf{The tail of size four,  $|T|=4$}
\end{center}
\end{tcolorbox}

We first state a necessary condition for a seaweed to be Frobenius.  We remind the reader that we are still in the case where
$b > n-c$.


\begin{theorem}\label{size4}
If $\mf{g} = \mf{p}_n^\D\left(\frac{a | b}{c}, \III\right)$ with $|T|=4$ is Frobenius, then 
\begin{itemize}
\item $a, b,$ and $c$ are odd, and
\item $v_{n-2}$ and $v_n$ are in the same component of $\mathcal{M}(\mathfrak{g})$.  
\end{itemize}
\end{theorem}

It follows 
that a Frobenius seaweed $\mf{p}_n^\D\left(\frac{a | b}{c}, \III\right)$ must satisfy $\gcd(a+b,b+c)=2$. However, this necessary condition is not sufficient to identify such a seaweed as Frobenius. See Example \ref{different}.

\begin{example}\label{different} In this example, we illustrate two meanders which are, respectively, associated with non-isomorphic seaweeds of the form $\mf{p}_n^\D\left(\frac{a | b}{c}, \III\right)$.  Both seaweeds are such that $\gcd(a+b,b+c)=2$.  However, one seaweed is Frobenius, while the other is not.  See Figures \ref{HomotopyTypeH11Frob} and \ref{HomotopyTypeH11NotFrob}, respectively.

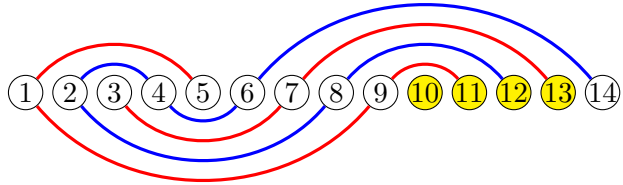
\begin{figure}[H]
\[\begin{tikzpicture}[scale=.59]

\vertex (1) at (1,0) {1};
\vertex (2) at (2,0) {2};
\vertex (3) at (3,0) {3};
\vertex (4) at (4,0) {4};
\vertex (5) at (5,0) {5};
\vertex(6) at (6,0) {6};
\vertex (7) at (7,0) {7};
\vertex (8) at (8,0) {8};
\vertex (9) at (9,0) {9};
\vertex[fill=yellow] (10) at (10,0) {10};
\vertex[fill=yellow] (11) at (11,0) {11};
\vertex[fill=yellow]  (12) at (12,0) {12};
\vertex[fill=yellow] (13) at (13,0) {13};
\vertex (14) at (14,0) {14};

\draw[color=red, line width=1.2 pt] (1) to [bend left=50] (5);
\draw[color=blue, line width=1.2 pt] (2) to [bend left=50] (4);
\draw[color=blue, line width=1.2 pt] (6) to [bend left=50] (14);
\draw[color=red, line width=1.2 pt] (7) to [bend left=50] (13);
\draw[color=blue, line width=1.2 pt] (8) to [bend left=50] (12);
\draw[color=red, line width=1.2 pt] (9) to [bend left=50] (11);
\draw[color=red, line width=1.2 pt] (1) to [bend right=50] (9);
\draw[color=blue, line width=1.2 pt] (2) to [bend right=50] (8);
\draw[color=red, line width=1.2 pt] (3) to [bend right=50] (7);
\draw[color=blue, line width=1.2 pt] (4) to [bend right=50] (6);

;\end{tikzpicture}\]
\caption{The seaweed $\mf{p}_{14}^\D \left( \frac{5 | 9}{9}, \III\right)$ has $\gcd(a+b,b+c)=2$ and index zero.}
\label{HomotopyTypeH11Frob}
\end{figure}

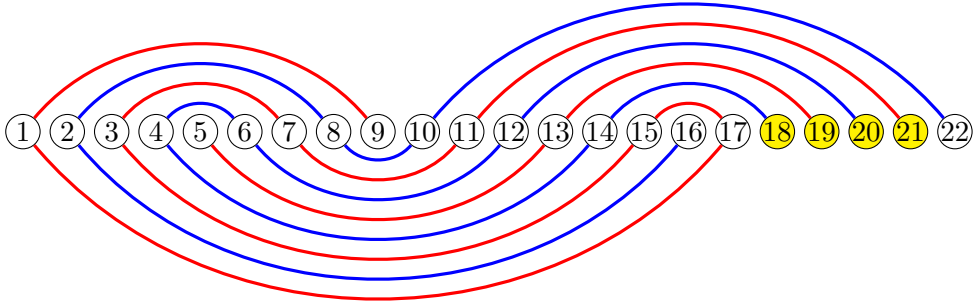
\begin{figure}[H]
\[\begin{tikzpicture}[scale=.59]

\vertex (1) at (1,0) {1};
\vertex (2) at (2,0) {2};
\vertex (3) at (3,0) {3};
\vertex (4) at (4,0) {4};
\vertex (5) at (5,0) {5};
\vertex (6) at (6,0) {6};
\vertex (7) at (7,0) {7};
\vertex (8) at (8,0) {8};
\vertex (9) at (9,0) {9};
\vertex (10) at (10,0) {10};
\vertex (11) at (11,0) {11};
\vertex (12) at (12,0) {12};
\vertex (13) at (13,0) {13};
\vertex (14) at (14,0) {14};
\vertex (15) at (15,0) {15};
\vertex (16) at (16,0) {16};
\vertex (17) at (17,0) {17};
\vertex[fill=yellow] (18) at (18,0) {18};
\vertex[fill=yellow] (19) at (19,0) {19};
\vertex[fill=yellow] (20) at (20,0) {20};
\vertex[fill=yellow] (21) at (21,0) {21};
\vertex (22) at (22,0) {22};

\draw[color=red, line width=1.2 pt] (1) to [bend left=50] (9);
\draw[color=blue, line width=1.2 pt] (2) to [bend left=50] (8);
\draw[color=red, line width=1.2 pt] (3) to [bend left=50] (7);
\draw[color=blue, line width=1.2 pt] (4) to [bend left=50] (6);
\draw[color=blue, line width=1.2 pt] (10) to [bend left=50] (22);
\draw[color=red, line width=1.2 pt] (11) to [bend left=50] (21);
\draw[color=blue, line width=1.2 pt] (12) to [bend left=50] (20);
\draw[color=red, line width=1.2 pt] (13) to [bend left=50] (19);
\draw[color=blue, line width=1.2 pt] (14) to [bend left=50] (18);
\draw[color=red, line width=1.2 pt] (15) to [bend left=50] (17);
\draw[color=red, line width=1.2 pt] (1) to [bend right=50] (17);
\draw[color=blue, line width=1.2 pt] (2) to [bend right=50] (16);
\draw[color=red, line width=1.2 pt] (3) to [bend right=50] (15);
\draw[color=blue, line width=1.2 pt] (4) to [bend right=50] (14);
\draw[color=red, line width=1.2 pt] (5) to [bend right=50] (13);
\draw[color=blue, line width=1.2 pt] (6) to [bend right=50] (12);
\draw[color=red, line width=1.2 pt] (7) to [bend right=50] (11);
\draw[color=blue, line width=1.2 pt] (8) to [bend right=50] (10);

;\end{tikzpicture}\]
\caption{The seaweed $\mf{p}_{22}^\D \left( \frac{9 | 13}{17}, \III\right)$ has $\gcd(a+b,b+c)=2$ and index two.}
\label{HomotopyTypeH11NotFrob}
\end{figure}

\end{example}

To differentiate between these, we make the following observations about the meanders illustrated in Figures \ref{HomotopyTypeH11Frob} and \ref{HomotopyTypeH11NotFrob}.

\begin{tcolorbox}[breakable, enhanced]

\begin{center}   
OBSERVATIONS
\end{center}

\begin{enumerate}

\item Each meander consists of two blue paths and two red paths; the blue paths span the even vertices, and the red paths span the odd vertices. 

\item In each meander, each red path has one end in the tail and therefore contributes nothing to the index of $\mf{p}_{n}^\D \left( \frac{a | b}{c}, \III\right)$. 

\item In the meander in Figure \ref{HomotopyTypeH11Frob}, each blue path has one end in the tail and therefore contributes nothing to the index of $\mf{p}_{n}^\D \left( \frac{a | b}{c}, \III\right)$.  However, the same does not hold for the meander in Figure \ref{HomotopyTypeH11NotFrob}.  Here, one blue path has zero ends in the tail, and the other has two ends in the tail; both components contribute one to the index of $\mf{p}_{n}^\D \left( \frac{a | b}{c}, \III\right)$. 

\item The top-bottom map gives two permutations associated to each meander: the first, $\sigma_1$ (the blue path), spans the even vertices and the second, $\sigma_2$ (the red path), spans the odd vertices.  For each meander, the permutation $\sigma_1$ is generated by $\Delta \equiv a+5 ~(\bmod ~n)$ by Theorem \ref{thm:mod}.  
\end{enumerate}

\noindent
In light of observations $3$ and $4$ above, we infer that what will differentiate the Frobenius case from the non-Frobenius case must be captured by the blue path.  We will add to the previously-mentioned greatest common divisor condition an argument similar to the proof of Theorem \ref{HomotopyTypeH1} to obtain a sufficient condition involving the Euler totient function. See Theorem \ref{HomotopyTypeH11}.  
\end{tcolorbox}

\begin{theorem}\label{HomotopyTypeH11}

If $\mathfrak{g}=\mf{p}_n^\D\left(\frac{a | b}{c}, \III\right)$ is such that $\gcd(a+b,b+c) = 2$, then $\mathfrak{g}$ is Frobenius precisely when 
$$0 < \xi \left(\frac{n}{2},\frac{\Delta}{2}  \right )< 0.5.$$ 

\end{theorem} 

\proof
The $\gcd$ condition must be satisfied by Theorem \ref{size4}.  

Consider $\sigma_1$ as in the fourth item in the observations above.  Divide each entry in $\sigma_1$ by two to obtain a permutation of $\{1, ...,\frac{n}{2}\}$.  Call this $\sigma$.  Since $\Delta$ generates $\sigma_1$, $\frac{\Delta}{2}$ generates $\sigma$.  Thus there are distinct $k_1, k_2 \in (0,\frac{n}{2})$ with 
\begin{eqnarray*}\label{tailofsize4proof}
\frac{n}{2}-1 + k_1 \frac{\Delta}{2} &\equiv & \frac{n}{2}-2 ~\left(\bmod ~\frac{n}{2}\right), \\
\frac{n}{2}-1 + k_2 \frac{\Delta}{2} &\equiv & 0 ~\left(\bmod ~\frac{n}{2}\right). 
\end{eqnarray*}
\noindent
If $k_1 > k_2$, then the path in the meander for $\mf{p}_n^\D\left(\frac{a | b}{c}, \III\right)$ containing $v_{n-2}$ also contains $v_n$.  In particular, if $k_1 > k_2$, then  $\mf{p}_n^\D\left(\frac{a | b}{c}, \III\right)$, is Frobenius.  These equations simplify to 


\begin{eqnarray*}
k_1 &\equiv & -\frac{\Delta}{2}^{\varphi (\frac{n}{2}) - 1} ~\left(\bmod ~\frac{n}{2}\right), \\
k_2 &\equiv & \frac{\Delta}{2}^{\varphi (\frac{n}{2}) - 1} ~\left(\bmod ~\frac{n}{2}\right). 
\end{eqnarray*} 
If $k_1>k_2$ and $0 \leq \frac{\Delta}{2}^{\varphi (\frac{n}{2}) - 1} ~(\bmod ~\frac{n}{2}) \leq \frac{n}{4}$, then $-\frac{\Delta}{2}^{\varphi (\frac{n}{2}) - 1} ~(\bmod ~\frac{n}{2}) > \frac{\Delta}{2}^{\varphi (\frac{n}{2}) - 1} ~(\bmod ~\frac{n}{2})$.  But $k_i \neq 0$ for $i = 1,2$, so $\Delta^{\varphi (\frac{n}{2}) - 1} ~(\bmod ~\frac{n}{2}) \neq 0$.  The result follows. \qed

\bigskip


Armed with Theorem \ref{HomotopyTypeH11}, we can now distinguish between the seaweeds whose meanders are illustrated in Figures \ref{HomotopyTypeH11Frob} and \ref{HomotopyTypeH11NotFrob}. Only one of the associated seaweeds is Frobenius. See Example \ref{distinguish2}.

\begin{example} \label{distinguish2}
Note that the seaweeds of Figures \ref{HomotopyTypeH11Frob} and \ref{HomotopyTypeH11NotFrob} both satisfy the gcd condition of Theorem \ref{HomotopyTypeH11}.
\begin{itemize}
\item The seaweed $\mf{p}_{14}^\D \left( \frac{5 | 9}{9}, \textrm{III}\right)$ is Frobenius:  $\frac{\Delta}{2} = 5$, $\varphi(7)=6$, and $\xi(7,5) \approx 0.21 <0.5 $.
\item
The seaweed $\mf{p}_{22}^\D \left( \frac{9 | 13}{17}, \textrm{III}\right)$ is not Frobenius: $\frac{\Delta}{2} = 7$, $\varphi (11) = 10$, and $\xi(11,7) \approx 0.86 > 0.5$.  
\end{itemize}
\end{example}


\noindent
\textbf{Question.} Using the notation above, is it possible to construct a meander such that 
the value of $\xi(\Delta,n)$ is any rational number between
0 and 1?

\section{Epilogue}\label{epilogue}
We close with a few remarks and a result announcement (see Theorem \ref{thm:main}) for the more advanced reader, who may be curious about seaweeds more generally.

As noted in the first footnote in the Introduction, a seaweed algebra can be described generally as a subalgebra of a simple Lie algebra $\mathfrak{g}$ -- of either classical or  exceptional type -- which is the intersection of two parabolic subalgebras of $\mathfrak{g}$ whose sum is $\mathfrak{g}$.  (A subalgebra is parabolic if it contains a maximal solvable (Borel) subalgebra $\mathfrak{g}$.)

Using this definition, one can construct a type-D seaweed which does not have seaweed ``shape'' and so cannot be described by pairs of partial compositions -- the latter being requisite for meander construction.
See Example \ref{nonseaweed shape}. 
However, the index of a type-D seaweed without seaweed shape corresponds to the index of a seaweed which does have seaweed shape -- 
so type-D seaweeds without seaweed shape can be defined by pairs of partial compositions and thus are subject to the index analysis of this article (see \textbf{\cite{DIndex}}).

\begin{example}\label{nonseaweed shape}  
In this example, we illustrate a type-D seaweed subalgebra of $\mathfrak{so}(10)$ reckoned as $\mathfrak{p}_1 \cap~ \mathfrak{p}_2$, where $\mathfrak{p}_1$ and $\mathfrak{p}_2$
are parabolic subalgebras of $\mathfrak{so}(10)$ whose sum is $\mathfrak{so}(10)$. See Figure \ref{nonseaweedshape}.
\end{example}

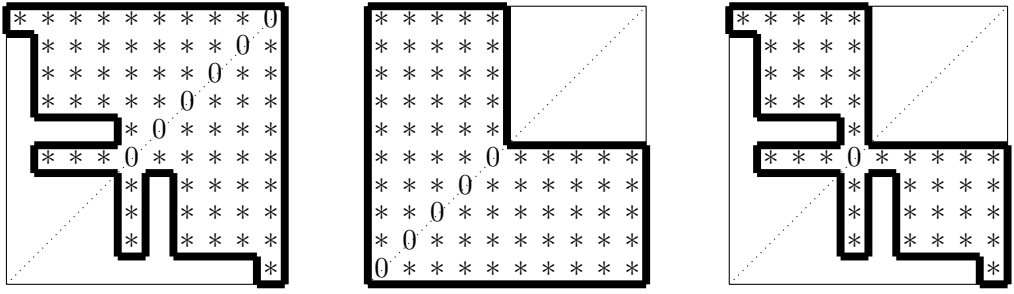
\begin{figure}[H]

\[\begin{tikzpicture}[scale=.37]
\draw (0,0) -- (0,10);
\draw (0,10) -- (10,10);
\draw (10,10) -- (10,0);
\draw (10,0) -- (0,0);

\draw [line width=3](0,10) -- (10,10);
\draw [line width=3](10,10) -- (10,0);
\draw [line width=3](0,10) -- (0,9);
\draw [line width=3](0,9) -- (1,9);
\draw [line width=3](1,9) -- (1,6);
\draw [line width=3](1,6) -- (4,6);
\draw [line width=3](4,6) -- (4,5);
\draw [line width=3](4,5) -- (1,5);
\draw [line width=3](1,5) -- (1,4);
\draw [line width=3](1,4) -- (4,4);
\draw [line width=3](4,4) -- (4,1);
\draw [line width=3](4,1) -- (5,1);
\draw [line width=3](5,1) -- (5,4);
\draw [line width=3](5,4) -- (6,4);
\draw [line width=3](6,4) -- (6,1);
\draw [line width=3](6,1) -- (9,1);
\draw [line width=3](9,1) -- (9,0);
\draw [line width=3](9,0) -- (10,0);

\draw [dotted] (0,0) -- (10,10);

\node at (0.5,9.4) {{\large *}};
\node at (1.5,9.4) {{\large *}};
\node at (2.5,9.4) {{\large *}};
\node at (3.5,9.4) {{\large *}};
\node at (4.5,9.4) {{\large *}};
\node at (5.5,9.4) {{\large *}};
\node at (6.5,9.4) {{\large *}};
\node at (7.5,9.4) {{\large *}};
\node at (8.5,9.4) {{\large *}};
\node at (9.5,9.6) {{$0$}};

\node at (1.5,8.4) {{\large *}};
\node at (2.5,8.4) {{\large *}};
\node at (3.5,8.4) {{\large *}};
\node at (4.5,8.4) {{\large *}};
\node at (5.5,8.4) {{\large *}};
\node at (6.5,8.4) {{\large *}};
\node at (7.5,8.4) {{\large *}};
\node at (8.5,8.6) {{$0$}};
\node at (9.5,8.4) {{\large *}};

\node at (1.5,7.4) {{\large *}};
\node at (2.5,7.4) {{\large *}};
\node at (3.5,7.4) {{\large *}};
\node at (4.5,7.4) {{\large *}};
\node at (5.5,7.4) {{\large *}};
\node at (6.5,7.4) {{\large *}};
\node at (7.5,7.6) {{$0$}};
\node at (8.5,7.4) {{\large *}};
\node at (9.5,7.4) {{\large *}};

\node at (1.5,6.4) {{\large *}};
\node at (2.5,6.4) {{\large *}};
\node at (3.5,6.4) {{\large *}};
\node at (4.5,6.4) {{\large *}};
\node at (5.5,6.4) {{\large *}};
\node at (6.5,6.6) {{$0$}};
\node at (7.5,6.4) {{\large *}};
\node at (8.5,6.4) {{\large *}};
\node at (9.5,6.4) {{\large *}};

\node at (4.5,5.4) {{\large *}};
\node at (5.5,5.6) {{$0$}};
\node at (6.5,5.4) {{\large *}};
\node at (7.5,5.4) {{\large *}};
\node at (8.5,5.4) {{\large *}};
\node at (9.5,5.4) {{\large *}};

\node at (1.5,4.4) {{\large *}};
\node at (2.5,4.4) {{\large *}};
\node at (3.5,4.4) {{\large *}};
\node at (4.5,4.6) {{$0$}};
\node at (5.5,4.4) {{\large *}};
\node at (6.5,4.4) {{\large *}};
\node at (7.5,4.4) {{\large *}};
\node at (8.5,4.4) {{\large *}};
\node at (9.5,4.4) {{\large *}};

\node at (4.5,3.4) {{\large *}};
\node at (6.5,3.4) {{\large *}};
\node at (7.5,3.4) {{\large *}};
\node at (8.5,3.4) {{\large *}};
\node at (9.5,3.4) {{\large *}};

\node at (4.5,2.4) {{\large *}};
\node at (6.5,2.4) {{\large *}};
\node at (7.5,2.4) {{\large *}};
\node at (8.5,2.4) {{\large *}};
\node at (9.5,2.4) {{\large *}};

\node at (4.5,1.4) {{\large *}};
\node at (6.5,1.4) {{\large *}};
\node at (7.5,1.4) {{\large *}};
\node at (8.5,1.4) {{\large *}};
\node at (9.5,1.4) {{\large *}};

\node at (9.5,.4) {{\large *}};

\end{tikzpicture}
\hspace{1cm}
\begin{tikzpicture}[scale=.37]
\draw (0,0) -- (0,10);
\draw (0,10) -- (10,10);
\draw (10,10) -- (10,0);
\draw (10,0) -- (0,0);

\draw [line width=3](0,10) -- (5,10);
\draw [line width=3](0,0) -- (0,10);
\draw [line width=3](5,5) -- (5,10);
\draw [line width=3](10,5) -- (5,5);
\draw [line width=3](10,5) -- (10,0);
\draw [line width=3](0,0) -- (10,0);

\draw [dotted] (0,0) -- (10,10);

\node at (0.5,9.4) {{\large *}};
\node at (1.5,9.4) {{\large *}};
\node at (2.5,9.4) {{\large *}};
\node at (3.5,9.4) {{\large *}};
\node at (4.5,9.4) {{\large *}};

\node at (0.5,8.4) {{\large *}};
\node at (1.5,8.4) {{\large *}};
\node at (2.5,8.4) {{\large *}};
\node at (3.5,8.4) {{\large *}};
\node at (4.5,8.4) {{\large *}};

\node at (0.5,7.4) {{\large *}};
\node at (1.5,7.4) {{\large *}};
\node at (2.5,7.4) {{\large *}};
\node at (3.5,7.4) {{\large *}};
\node at (4.5,7.4) {{\large *}};

\node at (0.5,6.4) {{\large *}};
\node at (1.5,6.4) {{\large *}};
\node at (2.5,6.4) {{\large *}};
\node at (3.5,6.4) {{\large *}};
\node at (4.5,6.4) {{\large *}};

\node at (0.5,5.4) {{\large *}};
\node at (1.5,5.4) {{\large *}};
\node at (2.5,5.4) {{\large *}};
\node at (3.5,5.4) {{\large *}};
\node at (4.5,5.4) {{\large *}};

\node at (.5,4.4) {{\large *}};
\node at (1.5,4.4) {{\large *}};
\node at (2.5,4.4) {{\large *}};
\node at (3.5,4.4) {{\large *}};
\node at (4.5,4.6) {{$0$}};
\node at (5.5,4.4) {{\large *}};
\node at (6.5,4.4) {{\large *}};
\node at (7.5,4.4) {{\large *}};
\node at (8.5,4.4) {{\large *}};
\node at (9.5,4.4) {{\large *}};

\node at (.5,3.4) {{\large *}};
\node at (1.5,3.4) {{\large *}};
\node at (2.5,3.4) {{\large *}};
\node at (3.5,3.6) {{$0$}};
\node at (4.5,3.4) {{\large *}};
\node at (5.5,3.4) {{\large *}};
\node at (6.5,3.4) {{\large *}};
\node at (7.5,3.4) {{\large *}};
\node at (8.5,3.4) {{\large *}};
\node at (9.5,3.4) {{\large *}};

\node at (.5,2.4) {{\large *}};
\node at (1.5,2.4) {{\large *}};
\node at (2.5,2.6) {{$0$}};
\node at (3.5,2.4) {{\large *}};
\node at (4.5,2.4) {{\large *}};
\node at (5.5,2.4) {{\large *}};
\node at (6.5,2.4) {{\large *}};
\node at (7.5,2.4) {{\large *}};
\node at (8.5,2.4) {{\large *}};
\node at (9.5,2.4) {{\large *}};

\node at (.5,1.4) {{\large *}};
\node at (1.5,1.6) {{$0$}};
\node at (2.5,1.4) {{\large *}};
\node at (3.5,1.4) {{\large *}};
\node at (4.5,1.4) {{\large *}};
\node at (5.5,1.4) {{\large *}};
\node at (6.5,1.4) {{\large *}};
\node at (7.5,1.4) {{\large *}};
\node at (8.5,1.4) {{\large *}};
\node at (9.5,1.4) {{\large *}};

\node at (.5,.6) {{$0$}};
\node at (1.5,.4) {{\large *}};
\node at (2.5,.4) {{\large *}};
\node at (3.5,.4) {{\large *}};
\node at (4.5,.4) {{\large *}};
\node at (5.5,.4) {{\large *}};
\node at (6.5,.4) {{\large *}};
\node at (7.5,.4) {{\large *}};
\node at (8.5,.4) {{\large *}};
\node at (9.5,.4) {{\large *}};

\end{tikzpicture}
\hspace{1cm}
\begin{tikzpicture}[scale=.37]
\draw (0,0) -- (0,10);
\draw (0,10) -- (10,10);
\draw (10,10) -- (10,0);
\draw (10,0) -- (0,0);

\draw [line width=3](0,10) -- (5,10);
\draw [line width=3](5,5) -- (5,10);
\draw [line width=3](10,5) -- (5,5);
\draw [line width=3](10,5) -- (10,0);
\draw [line width=3](0,10) -- (0,9);
\draw [line width=3](0,9) -- (1,9);
\draw [line width=3](1,9) -- (1,6);
\draw [line width=3](1,6) -- (4,6);
\draw [line width=3](4,6) -- (4,5);
\draw [line width=3](4,5) -- (1,5);
\draw [line width=3](1,5) -- (1,4);
\draw [line width=3](1,4) -- (4,4);
\draw [line width=3](4,4) -- (4,1);
\draw [line width=3](4,1) -- (5,1);
\draw [line width=3](5,1) -- (5,4);
\draw [line width=3](5,4) -- (6,4);
\draw [line width=3](6,4) -- (6,1);
\draw [line width=3](6,1) -- (9,1);
\draw [line width=3](9,1) -- (9,0);
\draw [line width=3](9,0) -- (10,0);

\draw [dotted] (0,0) -- (10,10);

\node at (0.5,9.4) {{\large *}};
\node at (1.5,9.4) {{\large *}};
\node at (2.5,9.4) {{\large *}};
\node at (3.5,9.4) {{\large *}};
\node at (4.5,9.4) {{\large *}};

\node at (1.5,8.4) {{\large *}};
\node at (2.5,8.4) {{\large *}};
\node at (3.5,8.4) {{\large *}};
\node at (4.5,8.4) {{\large *}};

\node at (1.5,7.4) {{\large *}};
\node at (2.5,7.4) {{\large *}};
\node at (3.5,7.4) {{\large *}};
\node at (4.5,7.4) {{\large *}};

\node at (1.5,6.4) {{\large *}};
\node at (2.5,6.4) {{\large *}};
\node at (3.5,6.4) {{\large *}};
\node at (4.5,6.4) {{\large *}};

\node at (4.5,5.4) {{\large *}};

\node at (1.5,4.4) {{\large *}};
\node at (2.5,4.4) {{\large *}};
\node at (3.5,4.4) {{\large *}};
\node at (4.5,4.6) {{$0$}};
\node at (5.5,4.4) {{\large *}};
\node at (6.5,4.4) {{\large *}};
\node at (7.5,4.4) {{\large *}};
\node at (8.5,4.4) {{\large *}};
\node at (9.5,4.4) {{\large *}};

\node at (4.5,3.4) {{\large *}};
\node at (6.5,3.4) {{\large *}};
\node at (7.5,3.4) {{\large *}};
\node at (8.5,3.4) {{\large *}};
\node at (9.5,3.4) {{\large *}};

\node at (4.5,2.4) {{\large *}};
\node at (6.5,2.4) {{\large *}};
\node at (7.5,2.4) {{\large *}};
\node at (8.5,2.4) {{\large *}};
\node at (9.5,2.4) {{\large *}};

\node at (4.5,1.4) {{\large *}};
\node at (6.5,1.4) {{\large *}};
\node at (7.5,1.4) {{\large *}};
\node at (8.5,1.4) {{\large *}};
\node at (9.5,1.4) {{\large *}};

\node at (9.5,.4) {{\large *}};

\end{tikzpicture}
\]
\caption{$\mathfrak{p_1}$, $\mathfrak{p_2}$, and $\mathfrak{p_1}\cap~ \mathfrak{p_2}$}
\label{nonseaweedshape}
\end{figure}

The classification of Frobenius Lie algebras remains an open question.  However,  Diatta and Manga (\textbf{\cite{DIATTA}}, 2014)  show that any Frobenius Lie algebra can be embedded into $\mathfrak{sl}(n)$ for some $n$.  They suggest that it would be of interest if there was a non-trivial obstruction to embedding the algebra as a seaweed.  In (\textbf{\cite{CAMCOLL}}, 2021), the first two authors, along with Hyatt and Magnant, have discovered such an obstruction as follows.  We need a bit of preliminary notation.

For a Frobenius Lie algebra $(\mathfrak{g}, [-,-])$, an index-realizing one-form $F$ (see (\ref{Frob})) is called a \textit{Frobenius one-form}, and then the natural map $\mathfrak{g} \rightarrow \mathfrak{g}^*$ defined by $x \mapsto  F([x,-])$ is an isomorphism.  The image of $F$ under the inverse of this map is called a \textit{principal element} of $\mathfrak{g}$ and will be denoted $\widehat{F}$.  It is the unique element of $\mathfrak{g}$ such that 

$$
F\circ \ad \widehat{F}= F([\widehat{F},-]) = F.  
$$

We have the following concise result which is applicable to all seaweeds -- including those of exceptional type.
\begin{theorem}[Cameron et al. \textbf{\cite{CAMCOLL}}, 2021]\label{thm:main}
If $\mathfrak{g}$ is a Frobenius seaweed and $\widehat{F}$ is a principal element of $\mathfrak{g}$, then the spectrum of $\ad \widehat{F}$ consists of an unbroken set of integers  
centered at one-half.  Moreover, the dimensions of the associated eigenspaces form a symmetric distribution. 
\end{theorem}

\begin{example}  To illustrate the content of Theorem \ref{thm:main}, consider the Frobenius seaweed $\mf{p}^\A_4\frac{2|2}{1|3}$, which has a spectrum consisting of the integers $\{-1,0,1,2\}$ with multiplicities as indicated in the multiset $\{-1^1,0^3,1^3,2^1    \}$.

\end{example}

\begin{example}
To illustrate the obstruction provided by Theorem \ref{thm:main}, consider the parametrized family of complex four-dimensional Frobenius Lie algebras with basis $\{e_1,e_2,e_3,e_4\}$ defined by 
the relations:  
$$
[e_1,e_4]=[e_2,e_3]=-e_1,~[e_2,e_4]=-e_3,~[e_3,e_4]=-e_3+ze_2,~~z\in\CC.
$$
The spectrum consists of the set $\left\{0,1,\frac{1\pm\sqrt{1-4z}}{2}\right\}$ and is an unbroken sequence of integers if and only if $z=0$ or $-2$.  So, by Theorem \ref{thm:main} almost all of the Lie algebras in this family cannot be embedded as seaweeds in an $\mathfrak{sl}(n)$.
See \textbf{\cite{CMR}}.
\end{example}

\bigskip
\noindent

\end{document}